\documentclass[12pt]{amsart} 
\usepackage{amsfonts,amsmath,amsthm,amsfonts,amscd,amssymb,amsmath,latexsym,multicol,euscript}
\usepackage[normalem]{ulem}
\usepackage{epsfig,graphicx,flafter,caption}
\usepackage{enumerate}
\usepackage[active]{srcltx}
\usepackage{hyperref,url}
\usepackage{color}
\usepackage{asymptote,asyfig}
\usepackage[small,nohug,UglyObsolete]{diagrams}
\diagramstyle[labelstyle=\scriptstyle]
\diagramstyle[noPostScript]
\usepackage{algorithm}
\usepackage[noend]{algpseudocode}
\usepackage{xr}
\usepackage{polynom}

\makeatletter

\def\jobis#1{FF\fi
  \def\preedicate{#1}%
  \edef\preedicate{\expandafter\strip@prefix\meaning\preedicate}%
  \edef\job{\jobname}%
  \ifx\job\preedicate
}

\makeatother

\if\jobis{proposal}%
 \def\try{subsection}%
\else
  \def\try{section}%
\fi

\theoremstyle{plain}
\newtheorem{theorem}{Theorem}[\try]
\newtheorem{corollary}[theorem]{Corollary}
\newtheorem{lemma}[theorem]{Lemma}

\newtheorem{example}[theorem]{Example}
\newtheorem{proposition}[theorem]{Proposition}
\newtheorem{definition-lemma}[theorem]{Definition-Lemma}
\newtheorem{definition-proposition}[theorem]{Definition-Proposition}
\newtheorem{definition-theorem}[theorem]{Definition-Theorem}

\newtheorem{definition}[theorem]{Definition}

\def\lfomitlist#1.#2.#3.#4.{{#1}_0,{#1}_1 #2 \dots #2\hat{{#1}_{#4}} #2\dots #2 {#1}_{#3}}

\def\alist#1.#2.#3.{{#1}_1 #2 {#1}_2 #2\dots #2 {#1}_{#3}}
\def\zlist#1.#2.#3.{#1_0 #2 #1_1 #2\dots #2 #1_{#3}}
\def\ltomitlist#1.#2.#3.{{#1}_0,{#1}_1 #2 \dots #2\hat {{#1}_i} #2\dots #2 {#1}_{#3}}
\def\lomitlist#1.#2.#3.{{#1}_0 #2 {#1}_1 #2 \dots #2 \hat {{#1}_i} #2 \dots #2 {#1}_{#3}}
\def\lmap#1.#2.#3.{#1 \overset{#2}{\longrightarrow} #3}
\def\mes#1.#2.#3.{#1 \longrightarrow #2 \longrightarrow #3}
\def\ses#1.#2.#3.{0\longrightarrow #1 \longrightarrow #2 \longrightarrow #3 \longrightarrow 0}
\def\les#1.#2.#3.{0\longrightarrow #1 \longrightarrow #2 \longrightarrow #3}
\def\res#1.#2.#3.{#1 \longrightarrow #2 \longrightarrow #3\longrightarrow 0}
\def\Hi#1.#2.#3.{\text {Hilb}^{#1}_{#2}(#3)}
\def\ten#1.#2.#3.{#1\underset {#2}{\otimes} #3}

\def\mderiv#1.#2.#3.{\frac {d^{#3} #1}{d #2^{#3}}}
\def\mfderiv#1.#2.#3.{\frac {\partial^{#3} #1}{\partial #2}}
\def\ggr#1.#2.#3.{\mathbb{G}_{#1}(#2,#3)}

\def\llist#1.#2.{{#1}_1,{#1}_2,\dots,{#1}_{#2}}
\def\ulist#1.#2.{{#1}^1,{#1}^2,\dots,{#1}^{#2}}
\def\lomitlist#1.#2.{{#1}_1,{#1}_2,\dots,\hat {{#1}_i}, \dots, {#1}_{#2}}
\def\lomitlistz#1.#2.{{#1}_0,{#1}_1,\dots,\hat {{#1}_i}, \dots, {#1}_{#2}}
\def\loc#1.#2.{\Cal O_{#1,#2}}
\def\fderiv#1.#2.{\frac {\partial #1}{\partial #2}}
\def\deriv#1.#2.{\frac {d #1}{d #2}}
\def\map#1.#2.{#1 \longrightarrow #2}
\def\rmap#1.#2.{#1 \dasharrow #2}
\def\emb#1.#2.{#1 \hookrightarrow #2}
\def\non#1.#2.{\text {Spec }#1[\epsilon]/(\epsilon)^{#2}}
\def\Hi#1.#2.{\text {Hilb}^{#1}(#2)}
\def\sym#1.#2.{\operatorname {Sym}^{#1}(#2)}
\def\Hb#1.#2.{\text {Hilb}_{#1}(#2)}
\def\Hm#1.#2.{\Hom_{#1}(#2)}
\def\prd#1.#2.{{#1}_1\cdot {#1}_2\cdots {#1}_{#2}}
\def\Bl #1.#2.{\operatorname {Bl}_{#1}#2}
\def\pl #1.#2.{#1^{\otimes #2}}
\def\mgn#1.#2.{\overline {M}_{#1,#2}}
\def\ialist#1.#2.{{#1}_1 #2 {#1}_2 #2 {#1}_3 #2\dots}
\def\pair#1.#2.{\langle #1, #2\rangle}
\def\gproj#1.#2.{\mathbb{P}_{#1}(#2)}
\def\gpr #1.#2.{\mathbb{P}^{#1}_{#2}}
\def\gaf #1.#2.{\mathbb{A}^{#1}_{#2}}
\def\vandermonde#1.#2.{\left|
\begin{matrix}
1 & 1 & 1 & \dots & 1\\
{#1}_1 & {#1}_2 & {#1}_3 & \dots & {#1}_{#2}\\
{#1}_1^2 & {#1}_2^2 & {#1}_3^2 & \dots & {#1}_{#2}^2\\
\vdots & \vdots & \vdots & \ddots & \vdots\\
{#1}_1^{#2-1} & {#1}_2^{#2-1} & {#1}_2^{#2-1} & \dots & {#1}_{#2}^{#2-1}\\
\end{matrix}
\right|
}
\def\vandermondet#1.#2.{\left|
\begin{matrix}
1 & {#1}_1   & {#1}_1^2 & \dots & {#1}_1^{#2-1}\\
1 & {#1}_2   & {#1}_2^2 & \dots & {#1}_2^{#2-1}\\
1 & {#1}_3   & {#1}_3^2 & \dots & {#1}_3^{#2-1}\\
\vdots & \vdots & \vdots & \ddots & \vdots\\
1 & {#1}_{#2}& {#1}_{#2}^2 & \dots & {#1}_{#2}^{#2-1}\\
\end{matrix}
\right|
}
\def\gr#1.#2.{\mathbb{G}(#1,#2)}
\def\bdd#1.#2.{{#1}_{\rfdown #2.}}

\def\ideal#1.{I_{#1}}
\def\ring#1.{\mathcal {O}_{#1}}
\def\fring#1.{\hat{\mathcal {O}}_{#1}}
\def\aring#1.{{\mathcal {O}}_{#1}^{\text{an}}}
\def\proj#1.{\mathbb {P}(#1)}
\def\pr #1.{\mathbb {P}^{#1}}
\def\dpr #1.{\hat{\mathbb {P}}^{#1}}
\def\af #1.{\mathbb{A}^{#1}}
\def\Hz #1.{\mathbb{F}_{#1}}
\def\Hbz #1.{\overline{\mathbb {F}}_{#1}}
\def\fb#1.{\underset {#1} {\times}}
\def\rest#1.{\underset {\ \ring #1.} \to \otimes}
\def\au#1.{\operatorname {Aut}\,(#1)}
\def\deg#1.{\operatorname {deg } (#1)}
\def\pic#1.{\operatorname {Pic}\,(#1)}
\def\pico#1.{\operatorname{Pic}^0(#1)}
\def\picg#1.{\operatorname {Pic}^G(#1)}
\def\n1#1.{\operatorname{N^1}(#1)}
\def\ner#1.{\operatorname{NS}(#1)}
\def\rdown#1.{\llcorner#1\lrcorner}
\def\rfdown#1.{\lfloor{#1}\rfloor}
\def\rup#1.{\ulcorner{#1}\urcorner}
\def\rfup#1.{\lceil{#1}\rceil}
\def\sship#1.{\langle{#1}\rangle}

\def\bp#1.{#1^{{}\leq 1}}
\def\rcup#1.{\lceil{#1}\rceil}
\def\cone#1.{\operatorname {NE}(#1)}
\def\mone#1.{\operatorname {NM}(#1)}
\def\none#1.{\operatorname {NF}(#1)}
\def\ccone#1.{\overline{\operatorname {NE}}(#1)}
\def\cmone#1.{\overline{\operatorname {NM}}(#1)}
\def\cnone#1.{\overline{\operatorname {NF}}(#1)}
\def\cbig#1.{\overline{\operatorname {B}(#1)}}
\def\coef#1.{\frac{(#1-1)}{#1}}
\def\vit#1.{D_{\langle #1 \rangle}}
\def\mm#1.{\overline {M}_{0,#1}}
\def\Hone#1.{H^1(#1,{\ring #1.})}
\def\ac#1.{\overline {\mathbb F}_{#1}}

\def\adj#1.{\frac {#1-1}{#1}}
\def\spn#1.{\overline{#1}}
\def\pek#1.#2.{\Cal P^{#1}(#2)}
\def\plk#1.#2.{\Cal P^{\leq #1}(#2)}
\def\ev#1.{\operatorname{ev_{#1}}}
\def\ilist#1.{{#1}_1,{#1}_2,\dotsc}
\def\bminv#1.{(\nu_1,s_1;\nu_2,s_2;\dots ;\nu_{#1},s_{#1};\nu_{r+1})}
\def\zinv#1.{(\nu_1,s_1;\nu_2,s_2;\dots ;\nu_{#1},s_{#1};0)}
\def\iinv#1.{(\nu_1,s_1;\nu_2,s_2;\dots ;\nu_{#1},s_{#1};\infty)}
\def\scr#1.{\mathbf{\EuScript{#1}}}
\def\mg#1.{\overline {M}_{#1}}
\def\inter#1.{\underset #1{\cdot}}
\def\cate#1.{\text{(\underline{#1})}}
\def\dls#1.{\overrightarrow{#1}}

\def\Hom{\operatorname{Hom}}

\def\Spec{\operatorname{Spec}}

\def\Proj{\operatorname{Proj}}

\def\dim{\operatorname{dim}}

\def\deg{\operatorname{deg}}

\def\Pic{\operatorname{Pic}}

\def\mult{\operatorname{mult}}
\def\mov{\operatorname{Mov}}

\def\rest{\operatorname{res}}

\def\Cox{\operatorname{Cox}}
\def\C`har{\operatorname{char}}

\def\nef{\operatorname{Nef}}

\def\C{\mathbb C}

\def\e{\Cal E}

\def\e1{E_1}
\def\e2{E_2}

\def\mapdown#1{\big\downarrow\rlap{$\vcenter
{\hbox{$\scriptstyle#1$}}$}}

\def\mapse#1{
{\vcenter{\hbox{$\mathop{\smash{\raise1pt\hbox{$\diagdown$}\!\lower7pt
\hbox{$\searrow$}}\vphantom{p}}\limits_{#1}\vphantom{\mapdown{}}$}}}}

\def\VR#1.{height#1pt&\omit&&\omit&&\omit&&\omit&&\omit&\cr}

\def\VRT#1.{height#1pt&\omit&&\omit&\cr}

\makeatletter
\renewcommand*\env@matrix[1][*\c@MaxMatrixCols c]{%
  \hskip -\arraycolsep
  \let\@ifnextchar\new@ifnextchar
  \array{#1}}
\makeatother

\begin{document}
\title{A geometric characterisation of toric varieties}
\author[M. Brown]{Morgan V.~Brown}
\address{Department of Mathematics\\
1365 Memorial Drive\\
Ungar 515\\
Coral Gables, FL 33146}
\email{mvbrown@math.miami.edu}
\author[J. M\textsuperscript{c}Kernan]{James M\textsuperscript{c}Kernan}
\address{Department of Mathematics\\
University of California, San Diego\\
9500 Gilman Drive \# 0112\\
La Jolla, CA  92093-0112, USA}
\email{jmckernan@math.ucsd.edu}
\author[R. Svaldi]{Roberto Svaldi}
\address{DPMMS\\ 
Centre for Mathematical Sciences\\
University of Cambridge\\
Wilberforce Road, Cambridge, CB3 0WB\\
United Kingdom}
\email{rsvaldi@dpmms.cam.ac.uk}
\author[H. Zong]{Hong R.~Zong}
\address{School of Mathematics\\
Institute of Advanced Study\\
F-114, Einstein Drive\\
Princeton, NJ 08540 USA}
\email{hrzong@math.ias.edu}

\thanks{The first author was partially supported by NSF research grant no: 0943832 and the
  first, second and third authors were partially supported by NSF research grant no:
  1200656 and no: 1265263.  The first author would like to thank UCSD for hosting him as a
  visitor during part of the completion of this work.  Part of this work was completed
  whilst the second and third authors were visiting the Freiburg Institute of Advanced
  Studies and they would like to thank Stefan Kebekus and the Institute for providing such
  a congenial place to work.  The third author would also like to thank MIT and both UCSD
  and Princeton University where he was a visitor when most of this work was completed.
  We are grateful to P. Cascini, A. Duncan, S. Keel, J. Koll\'ar, M.~Musta\c{t}\u{a},
  V. V. Shokurov and C. Xu, for some very helpful coments and suggestions.}

\begin{abstract} We prove a conjecture of Shokurov which characterises toric varieties
using log pairs.
\end{abstract}

\maketitle

\tableofcontents

\section{Introduction}
\label{s_introduction}

Toric varieties appear frequently in algebraic geometry.  This is surprising as the
definition of a toric variety is so restrictive; $X$ is normal and there is an open subset
isomorphic to a torus such that the action of the torus on itself extends to $X$.  On the
other hand the appearance of toric varieties is very useful as many geometric problems are
reduced to straightforward combinatorics.  We are interested in explaining why toric
varieties appear so often and to give additional criteria for their appearance.

One approach is to try to give a simple characterisation of toric varieties.  We give a
characterisation that only involves invariants coming from log pairs:
\begin{definition}\label{d_complexity} Let $X$ be a proper variety of dimension $n$
and let $(X,\Delta)$ be a log pair.  A \textbf{decomposition} of $\Delta$ is an expression
of the form
\[
\sum a_iS_i\leq \Delta,
\]
where $S_i\geq 0$ are $\mathbb{Z}$-divisors and $a_i\geq 0$, $1\leq i\leq k$.  The
\textbf{complexity} of this decomposition is $n+r-d$, where $r$ is the rank of the vector
space spanned by $\llist S.k.$ in the space of Weil divisors modulo algebraic equivalence
and $d$ is the sum of $\llist a.k.$.

The \textbf{complexity} $c=c(X,\Delta)$ of $(X,\Delta)$ is the infimum of the complexity
of any decomposition of $\Delta$.  
\end{definition}

Note that we don't require that the divisors $S_i$ are prime divisors (since the
components of $S_i$ might span a larger vector space).  On the other hand in practice the
smallest complexity is often achieved by taking $\llist S.k.$ to be prime divisors.  In
the special case when the coefficients of $D=\Delta=\sum S_i$ are all one, then $d$ is the
number of components of $D$.  It is well known that for a toric pair, that is, a toric
variety together with the sum of the invariant divisors, we have $d=n+r$, so that $c=0$.

We introduce some ad hoc but very convenient notation.  If 
\[
\Delta=\sum a_iD_i,
\]
is a boundary, that is, a divisor whose coefficients $a_i\in (0,1]$ then 
\[
\sship \Delta.=\sum_{i:a_i>1/2} D_i=\rfdown \Delta.+\rfup 2\Delta.-\rfdown 2\Delta..
\]

We give a characterisation of toric pairs involving the complexity:
\begin{theorem}\label{t_toric} Let $X$ be a proper variety of dimension $n$ and let 
$(X,\Delta)$ be a log canonical pair such that $-(K_X+\Delta)$ is nef.

If $\sum a_iS_i$ is a decomposition of complexity $c$ less than one then there is a
divisor $D$ such that $(X,D)$ is a toric pair, where $D\geq \sship \Delta.$ and all but
$\rdown 2c.$ components of $D$ are elements of the set $\{\, S_i \,|\, 1\leq i\leq k \,\}$.
\end{theorem}

\eqref{t_toric} is a special case of a conjecture of Shokurov, cf. \cite{Shokurov00},
which is stated in the relative case.  Here are two simple corollaries of \eqref{t_toric}:
\begin{corollary}\label{c_complexity} Let $X$ be a proper variety and let $(X,\Delta)$ 
be a log canonical pair such that $-(K_X+\Delta)$ is nef.

Then the complexity is non-negative.  
\end{corollary}

\begin{corollary}\label{c_span} Let $X$ be a proper variety of dimension $n$ and let 
$(X,\Delta)$ be a log canonical pair such that $-(K_X+\Delta)$ is nef.

If the complexity is less than one then the components of $\Delta$ span the N\'eron-Severi
group.
\end{corollary}

One can extend \eqref{t_toric} to the case of any field of characteristic zero:
\begin{corollary}\label{c_non} Let $k$ be a field of characteristic zero.  

Let $X$ be a proper variety over $k$ and let $(X,\Delta)$ be a log canonical pair such
that $-(K_X+\Delta)$ is nef.

If $\sum a_iS_i$ is a decomposition of complexity $c$ less than one then there is a
divisor $D$ such that $(X,D)$ is a toric pair, where $D\geq \sship \Delta.$ and all but
$\rdown 2c.$ components of $D$ are elements of the set $\{\, S_i \,|\, 1\leq i\leq k \,\}$.
\end{corollary}

We are able to prove that log pairs with small complexity have a simple birational
structure:
\begin{theorem}\label{t_form} Let $(X,\Delta)$ be a divisorially log terminal pair 
where $X$ is a $\mathbb{Q}$-factorial projective variety.

If $-(K_X+\Delta)$ is nef then we may find an ample divisor $A$ and a divisor
$0\leq \Delta_0\leq \Delta$ such that the numerical dimension of $K_X+A+\Delta_0$ is at
most the complexity of $(X,\Delta)$.

In particular if $\rmap X.Z.$ is the maximal rationally connected fibration then the
dimension of $Z$ is at most the complexity.  
\end{theorem}

Toric varieties are special as they are rational.  We are able to give a rationality
criterion in terms of a slightly different version of the complexity:
\begin{definition}\label{d_absolute} Let $X$ be a proper variety of dimension $n$
and let $(X,\Delta)$ be a log pair.  The \textbf{absolute complexity}
$\gamma=\gamma(X,\Delta)$ of $(X,\Delta)$ is $n+\rho-d$, where $\rho$ is the rank of the
group of Weil divisors modulo algebraic equivalence and $d$ is the sum of the coefficients
of $\Delta$.
\end{definition}

If $X$ is $\mathbb{Q}$-factorial then $\rho$ is the Picard number.  

\begin{theorem}\label{t_rational} Let $X$ be a proper variety.  Suppose that 
$(X,\Delta)$ is log canonical and $-(K_X+\Delta)$ is nef.

If $\gamma(X,\Delta)<\frac 32$ then there is a proper finite morphism $\map Y.X.$ of
degree at most two, which is \'etale outside a closed subset of codimension at least two,
such that $Y$ is rational.

In particular if $A_{n-1}(X)$ contains no $2$-torsion then $X$ is rational.
\end{theorem}
The condition on torsion in the class group is necessary and we give an example of this in
\S \ref{s_example}.  Note that most rationality criteria are used to establish
irrationality.  There are relatively few criteria to show rationality.

\eqref{t_toric} was proved for surfaces in \cite{KM99} for Picard number one (based
heavily on ideas of Shokurov) and in \cite{Shokurov00} in general.  Both proofs use
Shokurov's theory of complements.  Cheltsov, in unpublished work, proved \eqref{t_toric}
when $X$ is $\mathbb{Q}$-factorial projective and the Picard number is one.  The technique
he uses is the basis of our proof, which we will explain below.  \cite{Prokhorov01}
contains a proof of \eqref{t_toric} for threefolds in some special cases.  The method of
proof is to run the MMP.  \cite{Yao13} has a proof of \eqref{t_toric} when $X$ is a smooth
projective variety, $\Delta=\sum D_i$ has global normal crossings and $K_X+\Delta$ is
numerically trivial.  The method of proof is quite different from the other papers and
uses ideas coming from mirror symmetry and the powerful methods of Gross, Hacking, Keel
and Siebert, cf. \cite{GHKS16}.  \cite{Karzhemanov13} contains work related to both
\eqref{t_toric} and \eqref{t_rational}.

There are some examples to show \eqref{t_toric} and \eqref{t_rational} are sharp.  First
an example to show that not every invariant divisor is a component of $\Delta$:
\begin{example}\label{e_but} Consider $(X=\pr 2.,\Delta=L_1+L_2+1/2C)$ where $L_1$ and $L_2$ 
are two lines and $C$ is a conic, in general position.  Then $(X,\Delta)$ is divisorially
log terminal, $K_X+\Delta \sim_{\mathbb{Q}}0$ and the complexity is
\[
c=2+1-5/2=1/2.
\]
Note that $\sship \Delta.=L_1+L_2$.  Let $L_3$ be a third line in general position.  Then
$(\pr 2.,L_1+L_2+L_3)$ is a toric pair and two of the three invariant divisors are
components of $\Delta$ but not all three.
\end{example}

It is also not hard to see that it is crucial that $(X,\Delta)$ is log canonical:
\begin{example}\label{e_lc} Take $X=\Hz n.$ the unique $\pr 1.$-bundle over $\pr 1.$ 
with a curve $E_{\infty}$ of self-intersection $-n$.  Let $\Delta=2E_{\infty}+\sum F_i$,
where $\llist F.n+2.$ are $n+2$ fibres.  Then $K_X+\Delta\sim 0$ and the complexity
\[
c=2+2-(n+4)=-n,
\]
is arbitrarily large and negative.  Note that if one contracts $E_{\infty}$ then the image
of $\Delta$ is a boundary and the complexity is $c=1-n$.
\end{example}

One can also see that one cannot relax nef to pseudo-effective:

\begin{example}\label{e_pe} If we replace $\Delta$ by $E_{\infty}+\sum F_i$ in
\eqref{e_lc} then $(X,\Delta)$ is log canonical and $-(K_X+\Delta)$ is pseudo-effective
but the complexity is again $1-n$.
\end{example}

We also have an example where $X$ is smooth and all the coefficients are one:
\begin{example}\label{e_quadric} Let $Q=(XY-ZW=0)\subset \pr 4.$ be a rank four quadric
threefold.  Pick a small resolution $\map X.Q.$ with exceptional locus $L$ isomorphic to
$\pr 1.$.  Note that any hyperplane through the vertex of $Q$, which intersects the
quadric at infinity in two lines, intersects $Q$ in two planes through the vertex.  By
adjunction the sum of three such pairs is an element of $|-K_Q|$.

If $D=D_1+D_2+D_3+D_4+D_5+D_6$ is the strict transform of these six divisors then
$K_X+D\sim 0$ and the complexity
\[
c=3+2-6=-1.
\]
On the other hand three components of $D$ contain the curve $L$, so that $(X,D)$ is not
log canonical, even though $X$ is smooth and every component of $D$ has coefficient one.
\end{example}

It is also easy to see that we need to work with the absolute complexity for
unirrationality and that \eqref{t_toric} is sharp:
\begin{example}\label{e_elliptic} If $X=E$ is an elliptic curve and we take $\Delta=0$ then
$K_E \sim 0$ and the complexity is $1$.  On the other hand $E$ is not unirational.  Note
that the absolute complexity is $2$.
\end{example}

In fact if one works over a non-algebraically field, it is easy to see that we need to
allow an extension of degree two for rationality:
\begin{example}\label{e_real} If $C=V(x^2+y^2+z^2)$ is a smooth conic over $\mathbb{R}$
without a real point then we may find a divisor $D$ of degree one such that $K_X+D \sim 0$
so that the absolute complexity is one.  On the other hand $C$ is irrational but $C$
becomes rational if we replace $\mathbb{R}$ with $\mathbb{C}$.
\end{example}

We give an example in \S \ref{s_example} to show that we need a cover of degree two to
achieve rationality, cf. \eqref{t_rational}.  This example is in some sense a geometric
realisation of \eqref{e_real}.

Let us turn to a description of the proof of \eqref{t_toric}.  The first step is to
replace $(X,\Delta)$ by a divisorially log terminal model $(Y,\Gamma)$.  This means that
$Y$ is projective, $\mathbb{Q}$-factorial and $(Y,\Gamma)$ is divisorially log terminal.
There is a birational contraction map $\pi\colon\rmap Y.X.$ and the only exceptional
divisors have log discrepancy zero.  If $X$ is projective then we can take $\pi$ to be a
morphism and this is a standard reduction step (by a result of Hacon, see for example,
\cite[3.1]{KK09}).  If $X$ is not projective then there are examples which show it is not
always possible to arrange for $\pi$ to be a morphism.

For example, take $X$ to be any smooth proper variety which is not projective and take
$\Delta$ to be empty.  Let $\pi\colon\rmap Y.X.$ be a divisorially log terminal model of
$X$.  As $X$ is smooth it is kawamata log terminal and so $\pi$ is small.  $\pi$ is not
the identity morphism as $Y$ is projective and $X$ is not.  Therefore $\pi$ is not a
morphism as $X$ is $\mathbb{Q}$-factorial.

For a concrete example, consider the smooth toric threefold $X$ on page 71 of
\cite{Fulton93}.  It is not projective as it has no ample divisors.  It is easy to see
that if one flops an invariant curve $\rmap X.Y.$, corresponding to a diagonal edge of the
slanted faces of the tetrahedron, then $Y$ is projective and the induced birational map
$\pi\colon\rmap Y.X.$ is a divisorially log terminal model.

We prove the existence of divisorially log terminal models in \eqref{p_dlt}, contingent on
the existence of a nef divisor $M$ such that $K_X+\Delta+M$ is nef.  This covers the case
when either $K_X+\Delta$ or $-(K_X+\Delta)$ is nef and the latter is sufficient for our
purposes.  We check in \eqref{l_dlt} that the complexity of $(Y,\Gamma)$ is at most the
complexity of $(X,\Delta)$; this is straightforward since every exceptional divisor
extracted by $\pi$ is a component of $\Gamma$ of coefficient one.  Finally it is not hard
to see that it is enough to work with $(Y,\Gamma)$, cf. \eqref{l_persist}.

Thus we may assume that $X$ is projective, $\mathbb{Q}$-factorial and $(X,\Delta)$ is
divisorially log terminal.  The next step is to proceed based on the assumption that $X$
is a Mori dream space.  

To explain this step we first describe Cheltsov's argument which applies when the Picard
number is one.  In this case $K_X$ and all the components of $\Delta$ are proportional to
a very ample divisor $H$.  If we let $(Y,\Gamma)$ be the cone over $(X,\Delta)$ under the
embedding given by $H$ then $(Y,\Gamma)$ is log canonical and by construction every
component of $\Gamma$ is $\mathbb{Q}$-Cartier and passes through the vertex $p$ of the
cone.

The goal is then to prove \eqref{l_flipsabundance}, a local version of \eqref{t_toric}.
The proof of \eqref{l_flipsabundance} is based on the proof of \cite[18.22]{Kollaretal},
which establishes that the sum of the coefficients of $\Gamma$, which is precisely the sum
of the coefficients of $\Delta$, is no more than the dimension of $Y$.  Passing to a
composition of cyclic covers, we may assume that both $K_Y$ and every component of
$\Gamma$ is Cartier and in this case it suffices to check that $Y$ is smooth.  If we
replace components of $\Gamma$ whose coefficients sum to one by a general element of the
linear system they span we can apply adjunction and induction to conclude that $Y$ is
smooth.  Since the original variety is a quotient by a product of cyclic groups, it is not
hard to see that the original variety $Y$ is toric.  Since the only way to get a toric
variety as a cone is to start with a toric variety we see that $X$ must be toric; indeed
$X$ is isomorphic to the exceptional divisor of the blow up of the cone at the vertex $p$.

Unfortunately the naive generalisation of this argument does not apply if the Picard
number is not one.  The problem is that the cone over a variety of Picard number at least
two is not even $\mathbb{Q}$-factorial; for example the quadric cone which is the cone
over $\pr 1.\times \pr 1.$ is not $\mathbb{Q}$-factorial.  

Instead of working with a cone we work with the affine variety $Y$ associated to the Cox
ring of $X$.  $X$ is a Mori dream space if and only if the Cox ring is finitely generated.
The Cox ring is naturally graded by the class group, the group of Weil divisors modulo
linear equivalence.  As usual this grading corresponds to the action on $Y$ of an
algebraic group $H$, the spectrum of the group algebra associated to the class group,
which is the product of a torus and a finite abelian group.  We can recover $X$ as the
quotient of $Y$ by $H$.  In the case when the class group is isomorphic to $\mathbb{Z}$
(so that, in particular, the Picard number is one), $Y$ is a cone and $H$ is a one
dimensional torus, acting in the usual way on the lines of the cone.  As in the case of a
cone, there is a natural log pair $(Y,\Gamma)$ associated to $X$ and every component of
$\Gamma$ passes through the same point $p$.  $(Y,\Gamma)$ is log canonical if and only if
$(X,\Delta)$ is log canonical by \cite{Brown13}, \cite{GOST15}, and \cite{KO12}.  Mori
dream spaces were introduced in the very influential paper \cite{HK00}.  We actually use a
more sophisticated version of the Cox ring, which was introduced in \cite{Hausen08}.  It
has the advantage that every Weil divisor on $X$ corresponds to a Cartier divisor on $Y$,
so that we don't even need to take any cyclic covers.

The main point at this step of the proof is to bound the dimension of $Y$.  The dimension
of $Y$ is the dimension of $X$ plus the Picard number.  By assumption the sum of the
coefficients of a decomposition $\sum a_iD_i$ of $\Delta$ is at least the dimension of
$X$, minus one, plus the dimension $r$ of the space spanned by the components
$\llist D.k.$.  So we have to show that $r=\rho$, that is, the components $\llist D.k.$
generate the vector space of divisors modulo linear equivalence.

We prove this result by induction on $r$.  We start with the case that $\llist D.k.$ span
the same ray of the cone of divisors.  It is easy to show that the Picard number of $X$ is
one.  Consider for example the case that $X$ is a smooth projective surface and $K_X+D$ is
numerically trivial.  If the Picard number is not one then either there is a $-1$-curve
$\Sigma$ or a $\pr 1.$-bundle $\map X.C.$.  If $\Sigma$ is a $-1$-curve then $K_X$ is
negative on $\Sigma$ so that $D$ is positive on $\Sigma$.  As the components of $D$ are
proportional to each other it follows that every component of $D$ intersects $\Sigma$.  As
the sum of the coefficients of $D$ is at least three, $D\cdot \Sigma\geq 3$, which is
impossible as $K_X\cdot \Sigma=-1$.  If $\map X.C.$ is a $\pr 1.$-bundle and $\Sigma$ is a
general fibre we have $D\cdot\Sigma\geq 3$ and $K_X\cdot \Sigma=-2$, which is again
impossible.  In the general case we run an appropriate MMP.  After finitely many flips we
either get a divisorial contraction or a Mori fibre space and both cases we can rule out,
using a similar argument, cf. \eqref{l_sum}.

Otherwise we may pick two components $D_1$ and $D_2$ of $D$ such that neither
$P_1=m_1D_1-m_2D_2$ nor $P_2=m_2D_2-m_1D_1$ is pseudo-effective.  In this case consider
the $\pr 1.$-bundle given by the direct sum of the line bundles corresponding to $P_1$ and
$P_2$.  $Y$ is a Mori dream space and the two sections corresponding to $P_1$ and $P_2$
are contractible, $\rmap Y.Z.$.  In this case we proceed by induction on the rank $r$,
cf. \eqref{t_decomp}.  The details of this step are in \S \ref{s_local}.

To reduce to the case when $X$ is a Mori dream space we have to pass to a different model
$Y$ such that $-(K_Y+\Gamma)$ is ample for some kawamata log terminal pair $(Y,\Gamma)$.
Note that in this case $K_Y+B+\Gamma$ is numerically trivial, where $B=-(K_Y+\Gamma)$ is
ample.  So we look for divisors $0\leq \Delta_0\leq \Delta$ and ample divisors $A$ such
that $K_X+A+\Delta_0$ has numerical dimenson zero.  In this case $Y$ is a log terminal
model of $(X,A+\Delta_0)$.

If the numerical dimension is not zero then there is a non-trivial fibration $\map Y.Z.$.
Not every component of $D$ dominates $Z$, since otherwise the complexity of the general
fibre is less than zero, cf. \eqref{l_fibration}.  On the other hand it is not hard to
decrease the numerical dimension if there is a component of $D$ which does not dominate,
cf. \eqref{l_decrease_dim}.  To finish off, we replace $A+\Delta_0$ by a convex linear
combination of $A+\Delta_0$ and $M+\Delta$, where $M=-(K_X+\Delta)$, and cancel off common
components of $\Delta_0$ and exceptional divisors of $f\colon\rmap X.Y.$ so that the
complexity of $(X,A+\Delta_0)$ is close to the complexity of $(X,\Delta)$ and $f$ does not
contract any components of $\Delta$, cf. \eqref{l_minimal}.  The details are in \S
\ref{s_small}.

To finish the proof of \eqref{t_toric}, we need to know that if $Y$ is toric then so is
$X$.  The key point is to reduce to the case that $N=0$.  The first step is to pass to a
model such that no centre of $(X,\Delta)$ is contained in the exceptional locus of $f$,
cf. \eqref{l_up}.  We then perturb $\Gamma$ so that it is more singular along at least one
exceptional divisor, cf. \eqref{l_less}.  Taking a convex linear combination of
$\Delta_0$, a divisor supported on $N$ and the perturbed divisor we may decrease the
number of components of $N$ and we are done by induction on the number of components of
$N$.  The details are in \S \ref{s_reduction}.

Now we turn to the proof of \eqref{t_rational}.  The proof follows similar lines to the
proof of \eqref{t_toric}.  We may assume that $X$ is projective, $\mathbb{Q}$-factorial
and $(X,\Delta)$ is divisorially log terminal and by \eqref{c_numerical} we may assume
that $X$ is a Mori dream space.  If the absolute complexity is less than two then we can
conclude that the affine variety $Y$ associated to the Cox ring of $X$ has compound Du Val
singularities, meaning that there is a surface section with Du Val singularities.  If we
further assume that the absolute complexity is less than $3/2$ then we can conclude that
$Y$ has a compound $A_l$ singularity, meaning that a surface section of $Y$ has an $A_l$
singularity.

It follows that $Y$ is a hypersurface in affine space $\af m.$ given by a polynomial $q$
whose quadratic part has rank two.  The action of $H$ on $Y$ extends to $\af m.$.  The
quotient of $\af m.$ by $H$ is a toric variety and $X$ is birational to the image of $Y$
in this toric variety.  If $xy\in q$, that is, $xy$ is a monomial with non-zero
coefficient in $q$, then it is not hard to check that there is a one dimensional torus
whose general orbit intersects $X$ in a single point.  Thus $X$ is birational to an
invariant divisor so that $X$ is rational.  Otherwise after rescaling we may assume that
the quadratic part of $f$ has the form $x^2+y^2$.  If $x$ and $y$ have the same
multidegree then we may change variable and reduce to the previous case.  Otherwise there
must be torsion in the class group and there is a cover $\map Y.X.$ of degree two.  The
details are in \eqref{p_quadricmds}.

In \S \ref{s_example} we exhibit log canonical pairs $(X,\Delta)$ of absolute
complexity one such that $X$ is irrational.  The idea is to start with a conic bundle of
relative Picard number two over $\pr 1.\times \pr 1.$ and take a $\mathbb{Z}_2$-quotient
to achieve relative Picard number one.  The key observation is that the discriminant
curve, the locus of reducible fibres, makes no contribution in Kawamata's canonical bundle
formula.  Thus we can arrange for the discriminant curve to have arbitrarily large genus,
in which case $X$ is irrational.

We suspect that if the absolute complexity is less than two then we may always find a
cover so that $X$ is rational.  In this case we have to consider the extra possibility
that $Y$ has a compound singularity of type $D_l$, $E_6$, $E_7$, or $E_8$.  However we
were unable to see how to proceed in this case.

\makeatletter
\renewcommand{\thetheorem}{\thesubsection.\arabic{theorem}}
\@addtoreset{theorem}{subsection}
\makeatother

\section{Preliminaries}

In this section we will collect some definitions and preliminary results.  We work over a
field of characteristic zero which is algebraically closed unless otherwise stated.

\subsection{Notation and Conventions}
\label{s_notation}

Let $X$ be a proper variety.  $\rho(X)$ is the rank of the Picard group of $X$.  We denote
the class group, the group of Weil divisors modulo linear equivalence, by $A_{n-1}(X)$.

We will follow the terminology from \cite{KM98}.  In particular we only consider
valuations $\nu$ of $X$ whose centre on some birational model $Y$ of $X$ is a divisor.  A
\textit{log canonical place} of a log canonical pair $(X,\Delta)$ is any valuation $\nu$
whose log discrepancy is zero.

Suppose that $f\colon\rmap X.Y.$ is a rational map whose domain is an open subset $U$
whose complement has codimension at least two.  In this case if $D$ is an
$\mathbb{R}$-Cartier divisor on $Y$ we may define $f^*D$ as the $\mathbb{R}$-Weil divisor
whose restriction to $U$ is the usual pullback.

We say a proper morphism $f\colon\map X.Y.$ is a \textit{contraction morphism} if
$f_*\ring X.=\ring Y.$.  Let $f\colon\rmap X.Y.$ be a proper rational map of normal
quasi-projective varieties and let $p\colon\map W.X.$ and $q\colon\map W.Y.$ be a common
resolution of $f$.  We say that $f$ is a \textit{rational contraction} if $q$ is a
contraction morphism and the image of every $p$-exceptional divisor has codimension two or
more in $Y$.  We say that a prime divisor $P$ on $X$ is \textit{horizontal} if the image
of the generic point of $P$ is the generic point of $Y$.  We say that $P$ is
\textit{vertical} if it is not horizontal.

We say that a birational map $f\colon\rmap X.Y.$ is a \textit{birational contraction} if
$f$ is a rational contraction, so that every $p$-exceptional divisor is $q$-exceptional.
If $D$ is an $\mathbb{R}$-Cartier divisor on $X$ such that $D':=f_*D$ is
$\mathbb{R}$-Cartier then we say that $f$ is $D$-\textit{non-positive} (resp.
$D$-\textit{negative}) if we have $p^*D=q^*D'+E$ where $E\geq 0$ and $E$ is
$q$-exceptional (respectively $E$ is $q$-exceptional and the support of $E$ contains the
strict transform of the $f$-exceptional divisors).

Now suppose that $f\colon\rmap X.Y.$ is a birational contraction of projective varieties.
If $X$ is $\mathbb{Q}$-factorial and $(X,\Delta)$ is a divisorially log terminal pair such
that $f$ is $(K_X+\Delta)$-negative, $K_Y+\Gamma$ is nef and $Y$ is
$\mathbb{Q}$-factorial, where $\Gamma=f_*\Delta$, then we say that $f\colon\rmap X.Y.$ is
a \textit{log terminal model} of $K_X+\Delta$.  If the ring
\[
R(X,K_X+\Delta):=\bigoplus_{m\geq 0}H^0(X,\ring X.(m(K_X+\Delta)))
\] 
is finitely generated then $\rmap X.Z.$ is called the \textit{ample model} of $(X,\Delta)$, 
where 
\[
Z=\Proj R(X,K_X+\Delta).  
\]  

Let $D$ be an $\mathbb{R}$-Cartier divisor on a projective variety $X$.  Let $C$ be a
prime divisor.  If $D$ is big then
\[
\sigma_C(D)=\inf \{\, \mult_C(D') \,|\, D'\sim_\mathbb{R} D, D'\geq 0\,\}.
\]
More generally if $D$ is simply pseudo-effective we extend the definition of $\sigma_C$ as
follows.  Let $A$ be any ample $\mathbb{Q}$-divisor.  Following \cite{Nakayama04}, let
\[
\sigma_C(D)=\lim_{\epsilon\to 0} \sigma_C(D+\epsilon A).
\]
Then $\sigma_C(D)$ exists and is independent of the choice of $A$.  There are only
finitely many prime divisors $C$ such that $\sigma_C(D)>0$ and the $\mathbb{R}$-divisor
$N_\sigma(X,D)=\sum _C\sigma_C(D)C$ is determined by the numerical equivalence class of
$D$, cf. \cite[3.3.1]{BCHM10} and \cite{Nakayama04} for more details.   If we put 
\[
P_\sigma(X,D)=D-N_\sigma(X,D)
\]
then we will call 
\[
D=P_\sigma(X,D)+N_\sigma(X,D),
\]
\textit{Nakayama's Zariski decomposition}.  

Following \cite{Nakayama04} we define \textit{the numerical dimension}
\[
\kappa_{\sigma}(X,D)=\max_{H\in {\Pic}(X)}\{\, k\in \mathbb{N} \,|\, \limsup_{m\to \infty} \frac{h^0(X,\ring X.(mD+H))}{m^k} >0\,\},
\]
where $H$ is an ample divisor on $X$.  If $D$ is nef then this is the same as
\[
\nu(X,D)=\max \{\, k\in \mathbb{N} \,|\, H^{n-k}\cdot D^k>0 \,\}.  
\]

Let $f\in K[\llist x.n.]$ be a polynomial.  If $\mu$ is a monomial in $\llist x.n.$ then
we write $\mu\in f$ if and only if the coefficient of $\mu$ in $f$ is non-zero.

If $k$ is a field and $\bar k$ is the algebraic closure of $k$ then bars will denote
extension of schemes to $\bar k$.  

\subsection{Birational Geometry}

\begin{lemma}\label{l_orientation} Let $X$ be a $\mathbb{Q}$-factorial projective variety
and let $(X,\Delta)$ be a kawamata log terminal pair.  Suppose that $\Delta$ is big and
$K_X+\Delta$ is pseudo-effective.  Let $\pi\colon\rmap X.Z.$ be the ample model and let
$D$ be a prime divisor.  

Then $K_X+\Delta-dD$ is pseudo-effective for $d$ sufficiently small, if and only if,
either $D$ does not dominate $Z$ or the support of $D$ lies in the support of the stable
base locus of $K_X+\Delta$.
\end{lemma}
\begin{proof} Let $K_X+\Delta=P+N$ be Nakayama's Zariski decomposition.  Then the
components of $N$ are the prime divisors in the stable base locus of $K_X+\Delta$.

If the support of $D$ lies in the support of the stable base locus of $K_X+\Delta$ then we
may find $d>0$ such that $dD\leq N$ and in this case 
\[
K_X+\Delta-dD=P+(N-dD)\geq P
\] 
is pseudo-effective.

Let $H$ be the ample divisor on $Z$ corresponding to $K_X+\Delta$.  If $D$ does not
dominate $Z$ then we can pick $d>0$ and $H' \sim_{\mathbb{R}} H$ such that
$\pi^*H'\geq dD$, so that $K_X+\Delta-dD$ is pseudo-effective.

Now suppose $D$ dominates $Z$ and let $F$ be the general fibre of $\pi$.  Then $P|_F=0$.
Therefore if $K_X+\Delta-dD$ is pseudo-effective then $dD\leq N$.  But then the support of
$D$ lies in the support of the stable base locus of $K_X+\Delta$.
\end{proof}

We will need a version of the MMP for log canonical pairs.  

\begin{lemma}\label{l_mmp} Let $X$ be a $\mathbb{Q}$-factorial kawamata log terminal 
projective variety and let $(X,\Delta)$ be a log canonical pair.  

If $K_X+\Delta$ is not pseudo-effective then we may run the $(K_X+\Delta)$-MMP until we
arrive at a Mori fibre space.
\end{lemma}
\begin{proof} Pick an ample divisor $A$ such that $K_X+A+\Delta$ is not pseudo-effective.
Since $X$ is $\mathbb{Q}$-factorial kawamata log terminal we may find a divisor
$\Delta'\sim_{\mathbb{R}}A+\Delta$ such that $(X,\Delta')$ is kawamata log terminal.

In particular, \cite[1.3.3]{BCHM10} implies that the $(K_X+\Delta')$-MMP with scaling of
$A$ always terminates with a Mori fibre space.  On the other hand any run of the
$(K_X+\Delta')$-MMP with scaling of $A$ is automatically a run of the $(K_X+\Delta)$-MMP.
\end{proof}

We will need divisorially log terminal models in the case when $X$ is proper but not
necessarily projective.  In this case we need to relax the requirement that we have a
morphism.  To emphasize this point we use the term model rather than modification.

\begin{definition}\label{d_dlt} Let $(X,\Delta)$ be a log canonical pair, 
where $X$ is a proper variety.  

A \textbf{divisorially log terminal model} is a divisorially log terminal pair
$(Y,\Gamma)$, where $Y$ is a projective $\mathbb{Q}$-factorial variety, together with a
birational contraction $\pi\colon\rmap Y.X.$ such that
\[
K_Y+\Gamma=\pi^*(K_X+\Delta),
\]
and the only divisors contracted by $\pi$ have log discrepancy zero with respect to
$(X,\Delta)$. 
\end{definition}

We are only able to prove the existence of divisorially log terminal models in very
special cases:
\begin{proposition}\label{p_dlt} Let $(X,\Delta)$ be a log canonical pair where 
$X$ is a proper variety.  

If $M$ is a nef divisor such that $K_X+\Delta+M$ is nef then we may find a divisorially
log terminal model such that both $N=\pi^*M$ and $K_Y+\Gamma+\lambda N$, for some
$\lambda\geq 1$, are nef.

In particular if $M=\pm (K_X+\Delta)$ then $\pm(K_Y+\Gamma)$ is nef.
\end{proposition}

We will need some preliminary results, which are simple extensions of results by Shokurov,
cf. Addendum 4 of \cite{Shokurov06}.
\begin{lemma}\label{l_combinatorial} Let $\llist m.k.$ be positive real numbers and let 
$m$ and $r$ be positive integers.

Then there is a positive constant $\hbar$ such that if 
\[
a\in \{\, \sum \frac{a_i m_i}r  \,|\, \llist a.k.\in \mathbb{Z}, a_i\geq -m\,\}
\]
and $a>0$ then $a\geq \hbar$. 
\end{lemma}
\begin{proof} Clear.
\end{proof}

\begin{lemma}\label{l_decompose} Let $(X,\Delta)$ be a kawamata log terminal pair 
where $X$ is a $\mathbb{Q}$-factorial projective variety and $\Delta$ is a big
$\mathbb{R}$-divisor.

If $M$ is a nef $\mathbb{R}$-divisor then we may find a positive constant $\hbar$ with the
following property:

If $f\colon\rmap X.Y.$ is any sequence of $(K_X+\Delta)$-flips which are $M$-trivial and
$C$ is any curve spanning a $(K_Y+\Gamma)$-extremal ray of the cone of curves of $Y$ then
either $N\cdot C\geq \hbar$ or $N\cdot C=0$, where $N=f_*M$ and $\Gamma=f_*\Delta$.
\end{lemma}
\begin{proof} We may write 
\[
M=\sum m_iM_i,
\]
where $\llist m.k.$ are positive real numbers and $\llist M.k.$ are $\mathbb{Q}$-Cartier
divisors.  Pick this decomposition minimal with this property, so that $\llist m.k.$ are
independent over $\mathbb{Q}$.  Pick $M_i$ sufficiently close to $M$ so that we may find
\[
\Phi_i \sim_{\mathbb{R}} \Delta+M_i,
\]
where $(X,\Phi_i)$ is kawamata log terminal.  Pick a positive integer $r$ so that $rM_i$
is Cartier, for all indices $1\leq i\leq k$.   

We first check that all of these properties are preserved by any sequence
$f\colon\rmap X.Y.$ of $(K_X+\Delta)$-flips which are $M$-trivial.  By induction we are
reduced to the case of one flip.  If $R$ is the corresponding $(K_X+\Delta)$-extremal ray
then $R$ is spanned by a rational curve $C$.  As $M\cdot C=0$ and $\llist m.k.$ are
independent over $\mathbb{Q}$, we must have $M_i\cdot C=0$.  Thus $N$ is nef and $rN_i$ is
Cartier.  It is clear that
\[
N=\sum m_iN_i \qquad \text{and} \qquad \Psi_i=f_*\Phi_i \sim_{\mathbb{R}} \Gamma+N_i,
\]
since $f$ is an isomorphism in codimension one.  The pair $(Y,\Psi_i)$ is kawamata log
terminal as $f$ is a $(K_X+\Phi_i)$-flip.

Thus there is no harm in assuming that $f$ is the identity.  Suppose that $R$ is a
$(K_X+\Delta)$-extremal ray.  Then \cite[1]{Kawamata91} implies that $R$ is spanned by a
rational curve $C$ such that
\[
-2n\leq (K_X+\Phi_i)\cdot C\leq N_i\cdot C=\frac{a_i}r,
\]
for some integer $a_i$.  Now apply \eqref{l_combinatorial} with $m=2nr$.  
\end{proof}

\begin{proof}[Proof of \eqref{p_dlt}] As a first approximation, let $\pi\colon\map Y.X.$
be a log resolution of $(X,\Delta)$ such that $Y$ is projective.  We may write
\[
K_Y+\Gamma=K_Y+\widetilde \Delta+E=\pi^*(K_X+\Delta)+F
\]
where $\widetilde \Delta$ is the strict transform of $\Delta$, $E=\sum E_i$ is the sum of
the exceptional divisors and $F\geq 0$ is exceptional.  This model would be a divisorially
log terminal model provided $F=0$.  Our goal is to contract $F$ using the MMP, preserving
the condition that $N$ is nef.

We may write 
\[
K_Y+\Gamma+N=\pi^*(K_X+\Delta+M)+F.
\]
Note that $K_Y+\Gamma+N$ is pseudo-effective and the diminished base locus of
$K_Y+\Gamma+N$ is equal to the support of $F$.  Pick an ample divisor $A$ so that the
support of the stable base locus of $K_Y+A+\Gamma+N$ is equal to the support of $F$.  Then
the stable base locus of $K_Y+A+\Gamma+tN$ is equal to the support of $F$ for any
$t\geq 1$.

Let 
\[
\lambda=\max(1,\frac{2n}{\hbar})>0,
\] 
where $\hbar$ is defined in \eqref{l_decompose}.  Let $f\colon\rmap Y.Y'.$ be a step of
the $(K_Y+A+\Gamma+\lambda N)$-MMP with scaling of $A$.  If $R$ is the corresponding
extremal ray then $R$ is spanned by a rational curve $C$ such that
$(K_Y+A+\Gamma)\cdot C>-2n$ so that $N\cdot R=0$.  In particular $f_*N$ is nef.  If $f$
contracts a divisor then this divisor is a component of $F$ so that $f$ only contracts
divisors which are exceptional for $\pi$.  Therefore we are free to replace $Y$ by $Y'$.
Note that we might lose the property that $\pi$ is a morphism, when $f$ is a flip, but we
retain the property that $\pi$ is a birational contraction.

Now suppose that $g\colon\rmap Y.Y'.$ is a sequence of flips which are $N$-trivial.  By
\eqref{l_decompose} these are all steps of the $(K_Y+A+\Gamma+\lambda N)$-MMP with scaling
of $A$.  Since this MMP always terminates, after finitely many steps we construct a model
on which $F=0$.
\end{proof}

\subsection{Toric Geometry}

We say that $X$ is a \textit{toric variety} if $X$ is a normal variety over a field $k$
(not necessarily algebraically closed), there is a dense open subset $U$ isomorphic to
$\mathbb{G}^n_m$ such that the natural action of $U$ on itself extends to the whole of
$X$.  (Note that this is stronger than the usual definition in the literature which only
requires that $U$ is isomorphic to $\mathbb{G}^n_m$ after passing to the algebraic
closure).  We will say that a log pair $(X,D)$ is \textit{toric} if $X$ is a toric variety
and $D$ is the sum of the invariant divisors.

Every toric variety has a description in terms of fans.  We will use the notation of
\cite{Fulton93}. 

\begin{lemma}\label{l_missing} Let $X$ be a $\mathbb{Q}$-factorial projective toric variety
of dimension $n$ and let $V$ be a closed irreducible invariant subset.  Let $D$ be a fixed
invariant divisor.  

Then we may find a divisor $B\geq 0$ on $X$, supported on the invariant divisors which
contain neither $D$ nor $V$, such that $A=B|_V$ is very ample and every element of the
linear system $|A|$ lifts to $X$.
\end{lemma}
\begin{proof} We may as well assume that $D$ does not contain $V$.  If
$F\subset N_{\mathbb{R}}$ is the fan corresponding to $X$ then $D$ is given by a one
dimensional cone $\rho$ in $F$.  If $P_{\mathbb{R}}$ is the quotient vector space of
$N_{\mathbb{R}}$ corresponding to $V$ then the image of $\rho$ in $P_{\mathbb{R}}$ is
either a ray or zero.  Let $W$ be the closed invariant subset of $V$ determined by the
smallest cone which contains the image.

Let $A\geq 0$ be a very ample divisor on $V$ supported on the invariant divisors which do
not contain $W$.  $A$ determines a continuous piecewise linear function $f_A$ on
$P_{\mathbb{R}}$, which is non-negative as $A\geq 0$.  By composition we get a continuous
piecewise linear function $g$ on $N_{\mathbb{R}}$ which in turn corresponds to a divisor
$B$ supported on the invariant divisors.  $B\geq 0$ as $g$ is non-negative and the
restriction to $V$ is $A$, as $g$ is the composition of the natural projection and $f_A$.

It is enough to lift every invariant element $A'\in |A|$.  Note that, in the notation of
\cite{Fulton93}, $M(\sigma)\subset M$ is naturally the space of monomials on $V$, where
$\sigma$ is the cone corresponding to $V$.  We may find $u\in M(\sigma)$ such that
\[
A'=A+(\chi^u).
\]
On the other hand the zeroes and poles of $\chi^u$, as a rational function on $X$, don't
contain $V$.  Note that $f'=f+u$ is the continuous piecewise linear function corresponding
to $A'$ and $f'$ takes only non-negative values.  Then $g'=g+u$ is the composition of the
naturally projection and $f'$, and so $g'$ only take non-negative values.  Hence
\[
B'=B+(\chi^u)\in |B|
\]
is a divisor on $X$ which restricts to $A'$.  \end{proof}

We will need the next couple of results in the case when the groundfield is not
necessarily algebraically closed:
\begin{lemma}\label{l_persist} Let $k$ be any field.  Let $X$ and $Y$ be two proper 
varieties and let $(X,D)$ and $(Y,G)$ be two log pairs.  Let $\pi\colon\rmap X.Y.$ be a
birational contraction and $G=\pi_*D$.
\begin{enumerate} 
\item If $(X,D)$ is toric and $Y$ is projective then both $(Y,G)$ and $\pi$ are toric.
\item If $(Y,G)$ is toric, $X$ is projective and the exceptional divisors of $\pi$ are
components of $D$ that correspond to toric valuations of $Y$ then both $(X,D)$ and $\pi$
are toric.
\end{enumerate} 

\end{lemma}
\begin{proof} Suppose $(X,D)$ is toric.  If $H$ is an ample divisor on $Y$ then $\pi^*H$
is linearly equivalent to an invariant divisor.  As $Y=\Proj(X,\pi^*H)$ then both $(Y,G)$
and $\pi$ are toric.  This is (1).

Now suppose $(Y,G)$ is toric, the exceptional divisors of $\pi$ are components of $D$ and
correspond to toric valuations of $Y$.  We may find a toric pair $(Z,H)$ and a birational
morphism $f\colon\map Z.Y.$ whose only exceptional divisors correspond to these toric
valuations.  As the induced birational map $\rmap X.Z.$ is an isomorphism in codimension
one, it is a birational contraction.  Thus (2) follows from (1).
\end{proof} 

We will need an extension of \eqref{l_persist} to the case when $X$ and $Y$ are not
projective, only proper.  We start with:
\begin{lemma}\label{l_non} Let $(X,D)$ be a log pair over a field $k$ and let 
bars denote extension to the algebraic closure $\bar k$ of $k$. 

Then $(X,D)$ is toric if and only if $U=X-D$ is isomorphic to $\mathbb{G}_m^n$ and
$(\bar X,\bar D)$ is toric.
\end{lemma}
\begin{proof} One direction is clear.  

Otherwise if $U$ is a torus then it acts on itself and we get a morphism 
\[
\map U \times  U.U..
\]  
Now $U\times U$ is birational to $U\times X$ and so we get a rational map
\[
f\colon\rmap U \times  X.X..
\]
This induces a rational map 
\[
\bar f\colon\rmap \bar U \times  \bar X.\bar X..
\]
As $\bar X$ is toric, $\bar f$ is in fact a morphism.  But then $f$ is a morphism.
\end{proof}

\begin{lemma}\label{l_morphism} Let $k$ be any field.  Let $Y$ be a proper variety 
and let $\pi\colon\map Y.X.$ be a birational morphism of normal varieties.

If $Y$ is toric then both $X$ and $\pi$ are toric.
\end{lemma}
\begin{proof} We first prove this result using the additional hypothesis that $k$ is
algebraically closed.

Replacing $Y$ by a toric resolution, we may assume that $Y$ is
smooth and projective.  In particular $\pi$ is projective.  Let $U\subset Y$ be the torus.
By assumption there is an action
\[
\map U\times Y.Y. \qquad \text{given by} \qquad \map (u,y).u\cdot y..
\]
By composition there is a morphism
\[
f\colon\map U\times Y.X..
\]
Since $U\times X$ is birational to $U\times Y$ there is an induced rational map 
\[
g\colon\rmap U\times X.X..
\]
We check that $g$ is a morphism.   

Suppose that $y_1$ and $y_2$ are two points of $Y$, with the same image in $X$.  It
suffices to check that $f(u,y_1)=f(u,y_2)$ for all points $u\in U$.  As $\pi$ is
projective and birational and $X$ is normal the fibres of $\pi$ are connected.  Then $y_1$
and $y_2$ are connected by a chain of curves $C$ in $Y$ which are contracted by $\pi$.  As
the torus $U$ is connected the components of $C$ and of $u\cdot C$ are numerically
equivalent.  But then $u\cdot y_1$ and $u\cdot y_2$ belong to the connected curve
$u\cdot C$ which is contracted by $\pi$.  Thus $f(u,y_1)=f(u,y_2)$, for all $u\in U$ and
so there is an induced morphism $g$.

It is clear that $g$ defines an action of $U$ on $X$.  As $\pi$ is birational
$\map U.\pi(U).$ is an isomorphism.  Thus $X$ contains a torus and the natural action of
the torus extends to $U$.  Therefore $X$ is a toric variety.  

Now suppose that $k$ is not algebraically closed.  Let $U$ be the open subset of $Y$
isomorphic to $\mathbb{G}_m^n$.  As $\bar \pi \colon\map \bar Y.\bar X.$ is a toric
morphism the restriction of $\bar \pi$ to $\bar U$ is an isomorphism, so that
$\bar \pi(\bar U)$ is an open subset of $\bar X$.  But then the restriction of $\pi$ to
$U$ is an isomorphism and so $\pi(U)$ is an open subset of $X$ isomorphic to
$\mathbb{G}_m^n$.  It follows that $X$ is toric by \eqref{l_non} and it is easy to
conclude that $\pi$ is toric.
\end{proof}

We now return to assuming that the groundfield is algebraically closed.

\begin{lemma}\label{l_proper} Let $X$ be a proper variety and let $(X,\Delta)$ be a 
log canonical pair of complexity less than one such that $-(K_X+\Delta)$ is nef.  Suppose
that $(Y,\Gamma)$ is a divisorially log terminal model of $(X,\Delta)$,
$\pi\colon\rmap Y.X.$.

If $(Y,G)$ is a toric pair, where $G\geq \sship\Gamma.$ then $(X,D)$ is a toric pair,
where $D=\pi_*G\geq \sship\Delta.$.  
\end{lemma}
\begin{proof} It suffices to prove that $(X,D)$ is a toric pair.  Note that $(X,D)$ is log
canonical, $K_X+D$ is numerically trivial and 
\[
K_Y+G=\pi^*(K_X+D).
\]  
In particular a valuation $\nu$ is a log canonical place of $(Y,G)$ if and only if it is a
log canonical place of $(X,D)$.

Let $(Z,L)$ be a toric resolution of $(Y,G)$.  Then the exceptional divisors of
$\map Z.Y.$ have log discrepancy zero, so that the induced birational map $\rmap Z.X.$ is
a divisorially log terminal model of $(X,D)$.  Replacing $(Y,G)$ by $(Z,L)$ we may assume
that $Y$ is smooth.  Let $\map W.X.$ be a divisorially log terminal modification of
$(X,D)$.  We may write
\[
K_W+C=f^*(K_X+D),
\]
where $C$ is the strict transform of $D$ plus the exceptionals.  By \eqref{l_morphism} it
suffices to prove that $(W,C)$ is toric.  As $f$ only extracts divisors of log discrepancy
zero which also have log discrepancy zero for $(Y,G)$, possibly blowing up $Y$, we may
assume that the induced rational map $\rmap Y.W.$ is a birational contraction.  Replacing
$(X,D)$ by $(W,C)$ we may assume that $X$ is $\mathbb{Q}$-factorial and $X$ is kawamata
log terminal.

Let $\llist \nu.k.$ be the set of valuations corresponding to the exceptional divisors of
$\pi$.  Then the centres of $\llist \nu.k.$ are components of $G$ and so $\llist \nu.k.$
are log canonical places.  We may find a modification $f\colon\map W.X.$ such that the
exceptional divisors of $f$ are precisely the centres of $\llist \nu.k.$, where $W$ is
$\mathbb{Q}$-factorial and kawamata log terminal.  Replacing $(X,D)$ by $(W,C)$ once again
we may assume that $X$ is isomorphic to $Y$ in codimension one.

The result now follows by \cite[Corollary 2]{Fine89}.  
\end{proof}

\subsection{Calculus of the complexity}

In \S \ref{s_introduction} we defined the complexity $c(X,\Delta)$ and the absolute
complexity $\gamma(X,\Delta)$ for any log pair $(X,\Delta)$.  It is not hard to see that
the infimum is achieved for the complexity as there are only finitely many partitions of
the set of prime divisors contained in the support of $\Delta$.  It is immediate from the
definitions that
\[
c(X,\Delta) \leq \gamma(X,\Delta).
\]

\begin{lemma}\label{l_dlt} Let $X$ be a proper variety and let $(X,\Delta)$ be a log 
canonical pair.  

If $\pi\colon\rmap Y.X.$ is a divisorially log terminal model,
\[
K_Y+\Gamma=\pi^*(K_X+\Delta),
\]
then the complexity (respectively absolute complexity) of $(Y,\Gamma)$ is at most the
complexity (respectively absolute complexity) of $(X,\Delta)$.
\end{lemma}
\begin{proof} Let $\sum_{i=1}^m a_iS_i$ be a decomposition of $\Delta$.  Let $R_i$ be the
strict transform of $S_i$, $1\leq i\leq m$ and let $\llist E.k.$ be the exceptional
divisors.

 Let
\[
T_i=\begin{cases} R_i & \text{if $1\leq i\leq m$} \\  
                  E_{i-m} & \text{if $m<i\leq m+k$}
\end{cases}
\]
and
\[
b_i=\begin{cases} a_i & \text{if $1\leq i\leq m$} \\  
                  1   & \text{if $m<i\leq m+k$.}
\end{cases}
\]

Then $\sum b_iT_i$ is a decomposition of $\Gamma$.  The sum $e$ of the coefficients of
$\sum b_iT_i$ is $d+k$.  $\llist T.m+k.$, modulo algebraic equivalence, span a vector
space of dimension at most $r+k$.  Thus the complexity of the decomposition $\sum b_iT_i$
is at most
\[
n+(r+k)-(d+k)=n+r-d,
\]
which is the complexity of the decomposition given by $\sum a_iS_i$.  Thus the complexity
of $(Y,\Gamma)$ is at most the complexity of $(X,\Delta)$.  The absolute case is similar
and easier.
\end{proof}

\begin{definition}\label{d_local-complexity} Let $(x\in X,\Delta)$ be the germ of a log pair.  
A \textbf{local decomposition} of $\Delta$ is an expression of the form
\[
\sum a_iD_i\leq \Delta,
\]
where $D_i\geq 0$ are integral $\mathbb{Q}$-Cartier divisors and $a_i\geq 0$,
$1\leq i\leq k$.  The \textbf{local complexity} of this decomposition is $n-d$, where $n$
is the dimension of $X$ and $d$ is the sum of $\llist a.k.$.
\end{definition}

The following lemma establishes a local version of \eqref{t_toric}.  The proof is adapted
from the proof of \cite[18.22]{Kollaretal}:
\begin{lemma}\label{l_flipsabundance} Let $(x \in X,\Delta)$ be the germ of a log canonical 
pair where $X$ has dimension $n$ and let $\sum a_iD_i\leq \Delta$ be a local
decomposition.  Assume that $K_X$ and $\llist D.k.$ are Cartier.  

If $\gamma=n-\sum a_i=n-d$ is the local complexity then
\begin{enumerate}
\item $\gamma \geq 0$.
\item If $\gamma < 1$ then, possibly re-ordering $\llist D.k.$,
\[
(X,\alist D.+.m.)
\] 
is log smooth, where $m=n-\rfdown 2\gamma.$.   In addition 
\[
\sship \Delta.\leq \alist D.+.m..
\]  
\item If $\gamma < \frac 32$ then either $X$ is smooth at $x$ or has a $cA_l$ singularity
at $x$.
\end{enumerate}
\end{lemma}
\begin{proof} We proceed by induction on $n$.  All claims are clear for $n=1$ and so 
we assume that $n\geq 2$.  

Fix a log resolution $\pi\colon\map Y.X.$ of $(X,\Delta)$, with exceptional divisors
$\llist E.l.$.  Let $f$ be a general linear combination of $\llist g.k.$, the functions
defining $\llist D.k.$.  Let $S$ be the divisor cut out by $f$.  As $S$ specialises to
$D_i$, for each $i$, it follows that
\[
\mult_{E_j} S\leq \mult_{E_j} D_i.
\]
for each $1\leq i\leq k$ and $1\leq j\leq l$.  It also follows that $\pi$ is a log
resolution of $(X,\Delta+S)$.  For any $0 \leq b_i \leq a_i$ such that $\sum b_i=b\leq 1$,
it follows that the pair $(X,\Phi=bS+\sum (a_i-b_i)D_i)$ is log canonical, and the local
complexity of the indicated decomposition is $\gamma$.

Suppose that $0<b_i\neq a_i$, and $b=1$, so that
\[
\sum_{i:b_i\neq 0} a_i>1.
\]
Let $V\subset X$ be a codimension two subset.  As the pair $(X,\Delta)$ is log canonical
in a neighbourhood of the generic point of $V$ there is an index $i$ such that $b_i\neq 0$
and either $V$ is not contained in $D_i$ or $D_i$ is smooth at the generic point of $V$.
In this case $S$ is normal.  In particular if $d>1$ we may pick $\llist b.k.$ so that
$b=1$ and $S$ is normal.

As $S$ is Cartier and normal, $X$ is smooth in codimension two along $S$.  Therefore we
may write
\[
(K_X+\Phi)|_S=K_S+\Psi,
\]
where $(S,\sum (a_i-b_i)D_i|_S\leq \Psi)$ is log canonical and the local complexity is at
most $\gamma$.

Now suppose that $\gamma<1$.  As $n\geq 2$ then $d>1$ and so we may choose $\llist b.k.$
so that $S$ is normal.  By induction $S$ is smooth.  As $S$ is Cartier $X$ is smooth.
Then $\mult _x\Delta\leq n$ as $(X,\Delta)$ is log canonical.  In particular every
component of $\rdown\Delta.$ is smooth.  

(1) and (2) follow by induction on $n$.

Now suppose that $\gamma<3/2$.  If $n\geq 3$ then $d>1$.  By definition of compound
singularities it suffices to prove that $S$ has a $cA_l$ singularity.  By induction we may
assume that $n=2$ and we have to show that $X$ has an $A_l$ singularity.  As $K_X$ is
Cartier and $X$ is a normal surface, $X$ is Gorenstein.  As $\Delta\neq 0$ it follows that
$X$ is kawamata log terminal so that $X$ is canonical.  Thus $X$ has du Val singularities.
We may also assume that $\Delta=dD$, where $D=S$ is a prime Cartier divisor.

If $\pi\colon\map Y.X.$ is the minimal desingularisation of the surface $X$ then
$K_Y=\pi^*K_X$.  Let $G$ be the strict transform of $D$ and let $\llist E.l.$ be the
exceptional divisors.  Since $D$ is Cartier, we have
\[
f^*D=G+\sum m_iE_i
\]
where $\llist m.l.$ are positive integers.

The log discrepancy of $E_i$ with respect to $K_X+\Delta=K_X+dS$ is
\[
1-dm_i.
\]
As $(X,\Delta)$ is log canonical and $d>\frac 12$, we must have $m_i=1$, for all
$1\leq i\leq l$.  Hence
\[
0=f^*D\cdot E_j = (G +\sum E_i) \cdot E_j\geq \delta(E_j)-2
\]
where $\delta(E_j)$ is the degree of the vertex corresponding to $E_j$ in the dual graph
of the resolution.  It follows that every vertex in the dual graph has degree at most $2$
and so $X$ has an $A_l$ singularity.
\end{proof} 

\subsection{Mori Dream Spaces}

Recall, cf. \cite{HK00}, 
\begin{definition}\label{d_mds}  Let $X$ be a $\mathbb{Q}$-factorial normal projective variety.  We say
that $X$ is a \textbf{Mori dream space} if the following conditions hold:
\begin{enumerate}
\item X is $\mathbb{Q}$-factorial and $\Pic(X)_\mathbb{Q}= N^1(X)_\mathbb{Q}$;
\item the cone of nef divisors, $\nef(X)$, is the affine hull of finitely many semi-ample divisors;
\item there exist finitely many small birational maps $f_i\colon\rmap X.X_i.$, such that
each $X_i$ satisfies (1) and (2) and the closure of the cone of movable divisors,
$\mov(X)$, is the union of the cones $f_i^*\nef(X_i)$.
\end{enumerate}
\end{definition}

The Cox ring of a variety with finitely generated class group was originally defined in
\cite{HK00}; it is unique but ignores torsion in the class group.  Subsequently
\cite{Hausen08} gave a refined definition which takes into account torsion in the class
group.  As we would like to allow torsion we will use this definition of the Cox ring.  

We will need some of the basic properties of the Cox ring, cf. \cite{Hausen08} for more
details and proofs.  The most important result is that $X$ is a Mori dream space if and
only if the ring $R=\Cox(X)$ is a finitely generated $\mathbb{C}$-algebra.  One can use
this to give many examples of Mori dream spaces:
\begin{lemma}\label{l_mori} Let $X$ be a projective variety.  

The following are equivalent
\begin{enumerate} 
\item We may find a kawamata log terminal pair $(X,\Delta)$ such that 
$-(K_X+\Delta)$ is ample.  
\item We may find a kawamata log terminal pair $(X,\Delta)$ such that 
$-(K_X+\Delta)$ is big and nef.
\item We may find a kawamata log terminal pair $(X,\Delta)$ such that 
$K_X+\Delta$ is numerically trivial and $\Delta$ is big.  
\end{enumerate} 
In particular, if $X$ is $\mathbb{Q}$-factorial then $X$ is a Mori dream space.
\end{lemma}
\begin{proof} (1) clearly implies (2).  If $(X,\Delta)$ is kawamata log terminal and
$-(K_X+\Delta)$ is big and nef then $-(K_X+\Delta)$ is semiample.  Then we may find
$B\geq 0$, $B\sim_{\mathbb{R}} -(K_X+\Delta)$ such that $(X,\Delta+B)$ is kawamata log
terminal.  Thus (2) implies (3).

Suppose that $(X,\Delta)$ is kawamata log terminal, $K_X+\Delta$ is numerically trivial
and $\Delta$ is big.  We may find an ample $\mathbb{Q}$-divisor $A$ and a divisor
$B\geq 0$ such that
\[
\Delta \sim_{\mathbb{R}} A+B.
\]
Pick $\epsilon>0$ such that $(X,\Delta+\epsilon B)$ is kawamata log terminal.  Then 
\[
-(K_X+(1-\epsilon)\Delta+\epsilon B) \sim_{\mathbb{R}} \epsilon A.
\]
As $(X,(1-\epsilon)\Delta+\epsilon B)$ is kawamata log terminal, (3) implies (1).

The last assertion is \cite[1.3.2]{BCHM10}.
\end{proof}

If $X$ is a Mori dream space then let $Y=\Spec R$.  If $D$ is a prime divisor on $X$ one
can associate a Cartier divisor $G$ on $Y$.  The ring $R$ is naturally graded by the class
group $A_{n-1}(X)$.  There is a unique closed point $p\in Y$ corresponding to the unique
maximal homogeneous ideal and $p\in G$.  The grading corresponds to an action of the
diagonalisable group $H=\Spec \mathbb{C}[A_{n-1}(X)]$.  $X$ is a geometric quotient of $Y$
by the action of $H$ and the divisor $D$ is naturally the image of the associated Cartier
divisor $G$ on $Y$.

We will need a small strengthening of \cite[2.10]{HK00}:
\begin{theorem}\label{t_polynomial-ring} Let $X$ be a $\mathbb{Q}$-factorial projective
variety.

Then $X$ is toric if and only if the Cox ring is a polynomial ring generated by
$\dim X+\rho(X)$ variables, in which case the invariant divisors correspond to the
coordinate hyperplanes.
\end{theorem}
\begin{proof} If $X$ is a toric variety then the Cox ring is the homogeneous coordinate
ring of $X$ and the Cox ring is a polynomial ring with $\dim X+\rho(X)$ variables, which
correspond to the invariant divisors on $X$, cf. the discussion after the proof of
\cite[2.2]{Hausen08}.

Now suppose that the Cox ring is a polynomial ring.  Then $X$ is a Mori dream space.  In
particular its divisor class group $A_{n-1}(X)$ is a finitely generated abelian group and
the Cox ring is graded by the class group.  In this case $X$ is the GIT quotient of affine
space $\af m.$ by a diagonalisable group $H$, the product of a torus and a finite abelian
group, \cite[2.2]{Hausen08}.  Therefore $X$ is a toric variety.
\end{proof}

We will also need:
\begin{lemma}\label{l_cartier} Let $X$ be a $\mathbb{Q}$-factorial projective variety.
Suppose that $X$ is a Mori dream space and let $R=\Cox (X)$ be the Cox ring.

If $Y=\Spec R$ then $K_Y$ is Cartier.  In particular if $Y$ is Cohen-Macaulay then $Y$ is
Gorenstein.
\end{lemma}
\begin{proof} Let $H=\Spec \mathbb{C}[A_{n-1}(X)]$.  According to \cite[2.2]{Hausen08} it
suffices to check that $K_Y$ is $H$-invariant.  It also follows from \cite[2.2]{Hausen08}
that there is a universal $H$-torsor $q\colon\map \hat X.X.$ and it suffices to prove that
$K_{\hat X}$ is $H$-invariant.

The group $H$ decomposes as a torus and a finite abelian group.  The morphism $q$ then
decomposes as a torus bundle followed by an \'etale cover.  It follows that $K_{\hat X}=q^*K_X$ 
so that $K_{\hat X}$ is $H$-invariant.
\end{proof}

\makeatletter
\renewcommand{\thetheorem}{\thesection.\arabic{theorem}}
\@addtoreset{theorem}{section}
\makeatother

\section{Local to global}
\label{s_local}

\begin{theorem}\label{t_mdscase} Let $X$ be a $\mathbb{Q}$-factorial projective variety 
with kawamata log terminal singularities and let $(X,\Delta)$ be a log canonical pair.
Suppose that $-(K_X+\Delta)$ is nef and $\sum a_iS_i$ is a decomposition of complexity $c$
less than one for $\Delta$.

If $X$ is a Mori dream space then there is a divisor $D$ such that $(X,D)$ is a toric
pair, where $D\geq \sship \Delta.$ and all but $\rdown 2c.$ components of $D$ are elements
of the set $\{\, S_i \,|\, 1\leq i\leq k \,\}$.
\end{theorem}

\begin{theorem}\label{t_decomp} Let $X$ be a $\mathbb{Q}$-factorial kawamata log terminal 
projective variety.  Suppose that $(X,\Delta)$ is a log canonical pair such that
$K_X+\Delta$ is numerically trivial.  Let $\sum a_i S_i$ be a decomposition of $\Delta$
with complexity less than $1$.

If $X$ is a Mori dream space then $\llist S.k.$ generate $A_{n-1}(X)_{\mathbb{Q}}$.
\end{theorem}

\begin{lemma}\label{l_sum} Let $X$ be a $\mathbb{Q}$-factorial kawamata log terminal 
projective variety of dimension $n$ and let $(X,\Delta)$ be a log canonical pair.  Let
$D=\sum a_iS_i\leq \Delta$ be a decomposition of $\Delta$.

If 
\begin{enumerate} 
\item $K_X+\Delta$ is numerically trivial,
\item $d=\sum_{i=1}^k a_i>n$, and 
\item $\llist S.k.$ all span the same ray of the cone of effective divisors
\end{enumerate} 
then the Picard number of $X$ is one.  
\end{lemma}
\begin{proof} Let $\Theta=\Delta-D$.  We run the $(K_X+\Theta)$-MMP with scaling of some
ample divisor.

Let $f\colon\rmap X.Y.$ be a step of this MMP.  $f$ is $D$-positive and as the components
of $S$ span the same ray of the cone of effective divisors, it follows that $f$ is
$S_i$-positive, for every $1\leq i\leq k$.  Let $T_i=f_*S_i$.

Suppose that $f$ is a divisorial contraction.  If $V$ is the image of the exceptional
divisor $E$ then $T_i$ contains $V$.  If $\Gamma=f_*\Delta$ then $(Y,\Gamma)$ is log
canonical and the local complexity about a point of $V$ is negative.  This is not possible
by (1) of \eqref{l_flipsabundance}.

If $f$ is a flip then $\rho(X)=\rho(Y)$ and $\llist T.k.$ all span the same ray of the
cone of effective divisors.  We replace $X$ by $Y$ in this case.  \eqref{l_mmp} implies
that after finitely many flips $f$ must be a Mori fibre space.  Let $F$ be the general
fibre and let $\Sigma$ be the restriction of $\Delta$ to $F$.  Then $(F,\Sigma)$ is log
canonical.  As $\llist S.k.$ dominate $Y$, the sum of the coefficients of $\Sigma$ is
greater than $n$.  \eqref{l_flipsabundance} implies that $F$ has dimension $n$.  But then
$Y$ is a point and $X$ has Picard number one.
\end{proof} 

\begin{proof}[Proof of \eqref{t_decomp}] We proceed by induction on the dimension $r$ of
the span of $\llist S.k.$ in $A_{n-1}(X)_{\mathbb{Q}}$.  If $r=1$ then we may apply
\eqref{l_sum}.

Otherwise, we may assume that $S_1$ and $S_2$ are linearly independent in
$A_{n-1}(X)_{\mathbb{Q}}$.  Pick integers $m_1$ and $m_2$ such that $m_1S_1$ and $m_2S_2$
are Cartier, and neither $m_1S_1-m_2S_2$ nor $m_2S_2-m_1S_1$ is pseudo-effective.

Consider the $\pr 1.$-bundle 
\[
Y=\proj{\ring X.(m_1D_1)\oplus \ring X.(m_2D_2)}..
\]
Let $f\colon\map Y.X.$ be the structure morphism.  Then $Y$ is a $\mathbb{Q}$-factorial
projective variety with kawamata log terminal singularities.  There are two distinguished
sections, which we will call $E_0$ and $E_\infty$.  Set
$\Gamma=f^* \Delta+ E_0 + E_{\infty}$.  Adjunction implies that $(Y,\Gamma)$ is a log
canonical pair, and that $K_Y+\Gamma$ is numerically trivial.  Note that
$\rho(Y)=\rho(X)+1$.  Finally, $Y$ is a Mori Dream Space because the Cox ring of $Y$ is
isomorphic as a ring to the Cox ring of $X$ with two variables adjoined, corresponding to
the sections $E_0$ and $E_{\infty}$, cf. \cite[3.2]{Brown13}.

As both $m_1S_1-m_2S_2$ and $m_2S_2-m_1S_1$ are not pseudo-effective, $E_0|_{E_0}$ and
$E_{\infty}|_{E_\infty}$ are not pseudoeffective.  Thus $D=E_0+E_{\infty}$ has Kodaira
dimension zero.  As $Y$ is a Mori dream space we may run $g\colon\rmap X.Z.$ the $D$-MMP
and the image of $D$ is semiample.  Thus the birational map $g$ contracts $E_0$ and
$E_{\infty}$.  $Z$ is a $\mathbb{Q}$-factorial projective variety with kawamata log
terminal singularities.

Note that $K_Z+\Psi$ is numerically trivial and $(Z,\Psi=g_*\Gamma)$ is log canonical.  
Note that 
\[
\rho(Z)=\rho(Y)-2=\rho(X)-1.
\]

If $T_i=f^*S_i$ then the dimension of the space spanned by $\llist T.k.$ in
$A_{n-1}(Y)_{\mathbb{Q}}$ is equal to $r$.  Let $C_i=g_*T_i$.  As $m_1C_1-m_2C_2$ is
linearly equivalent to zero in $Z$, $\llist C.k.$ span a vector space of dimension $r-1$
in $A_{n-1}(Z)_{\mathbb{Q}}$.

$Z$ is a Mori dream space, as $Y$ is a Mori dream space.  By induction $\llist C.k.$
generate $A_{n-1}(Z)_{\mathbb{Q}}$.  We can identify $A_{n-1}(Z)_{\mathbb{Q}}$ with
\[
\frac{A_{n-1}(X)_{\mathbb{Q}}}{\langle m_1S_1-m_2S_2\rangle}
\]
and so $\llist S.k.$ span $A_{n-1}(X)_{\mathbb{Q}}$.  \end{proof} 

\begin{proof}[Proof of \eqref{t_mdscase}] Since $-(K_X+\Delta)$ is nef, and $X$ is a Mori
dream space, we can find $B\sim_{\mathbb{R}} -(K_X+\Delta)$ such that $(X,\Delta+B)$ is
log canonical.  By \eqref{t_decomp}, the components of any decomposition of $(X,\Delta+B)$
with complexity less than $1$ generate $A_{n-1}(X)_{\mathbb{Q}}$.  The pair $(X,\Delta)$
has a decomposition with complexity less than $1$, and this is a decomposition of
$(X,\Delta+B)$.  Thus $r=\rho$, where $r$ is the rank of the group generated by the
$\llist S.k.$ and $\rho$ is the Picard number.

Let $Y=\Spec R$ where $R=\Cox(X)$ is the Cox ring.  Then $Y$ has dimension $n+\rho$.  Let
$T_i$ be the divisor corresponding to $S_i$ and let $\Gamma=\sum a_iT_i$.  Then $T_i$ is a
Cartier divisor and every component $\llist T.k.$ contains the point $p$ corresponding to
the unique maximal ideal which is homogeneous, \cite[2.2]{Hausen08}.  By \cite[1.1]{KO12}
the pair $(Y,\Gamma)$ is log canonical (as observed in \cite[2.5]{KO12} their result
applies to the Cox ring, as defined in \cite{Hausen08}).

\eqref{l_flipsabundance} implies that $Y$ is smooth, every component of $\Gamma$ of
coefficient greater than $1/2$ is smooth and at least $\dim Y-\rfdown 2c.$ components of
$\llist T.k.$ are smooth at $p$ and intersect transversally.

The result now follows from \eqref{t_polynomial-ring}.  
\end{proof}

\section{Log divisors of small numerical dimension}
\label{s_small}

\begin{proposition}\label{p_numerical} Assume \eqref{t_toric}$_{n-1}$, that is, 
assume \eqref{t_toric} when $X$ has dimension $n-1$.

Let $(X,\Delta)$ be a divisorially log terminal pair where $X$ is a $\mathbb{Q}$-factorial
projective variety of dimension $n$.

If $-(K_X+\Delta)$ is nef then we may find an ample divisor $A$ and a divisor
$0\leq \Delta_0\leq \Delta$ such that $K_X+A+\Delta_0$ is pseudo-effective, no component
of $N_{\sigma}(X,K_X+A+\Delta_0)$ is a component of $\Delta_0$, and the numerical
dimension of $K_X+A+\Delta_0$ is at most the complexity of $(X,\Delta)$.  
\end{proposition}

\begin{corollary}\label{c_numerical} Assume \eqref{t_toric}$_{n-1}$. 

Let $X$ be a $\mathbb{Q}$-factorial projective variety of dimension $n$.  Suppose
$(X,\Delta)$ is a divisorially log terminal pair such that $-(K_X+\Delta)$ is nef.
Let $\gamma_0\in (0,2)$.  

If the absolute complexity $\gamma(X,\Delta)<\gamma_0$ then there is a log canonical pair
$(Y,\Gamma)$ such that $-(K_Y+\Gamma)$ is ample, $\gamma(Y,\Gamma)<\gamma_0$ and $Y$ is a
$\mathbb{Q}$-factorial projective variety birational to $X$.
\end{corollary}

\begin{lemma}\label{l_fibration} Assume \eqref{t_toric}$_{n-1}$.  

Let $(X,\Delta)$ be a divisorially log terminal pair where $X$ is a $\mathbb{Q}$-factorial
projective variety of dimension $n$.  Let $A$ be an ample divisor such that $K_X+A+\Delta$
is pseudo-effective and let $\phi\colon\rmap X.Z.$ be the ample model of $K_X+A+\Delta$.
Assume that no component of $N=N_{\sigma}(X,K_X+A+\Delta)$ is a component of $\Delta$.

If the dimension of $Z$ is greater than the complexity of $(X,\Delta)$ then we may find a
component $P$ of $\Delta$ which is vertical.
\end{lemma}
\begin{proof} Let $f\colon\rmap X.Y.$ be a log terminal model of $K_X+A+\Delta$.  Then
there is a contraction morphism $g\colon\map Y.Z.$.  The divisors contracted by $f$ are
the components of $N$ and so $f$ does not contract any components of $\Delta$.  If
$B=f_*A$ and $\Gamma=f_*\Delta$ then $(Y,B+\Gamma)$ is divisorially log terminal.

If $F$ is the general fibre of $g$ and $\Theta$ is the restriction of $B+\Gamma$ to $F$
then $(F,\Theta)$ is log canonical and $K_F+\Theta$ is numerically trivial.  Let
$\sum a_iS_i$ be a decomposition of $(X,\Delta)$ which computes the complexity.  Let $C_i$
be the restriction to $F$ of the image of $S_i$.  Then $\sum a_iC_i$ is a decomposition of
$(F,\Theta)$, where the sum ranges over the indices $i$ such that at least one component
of $S_i$ is horizontal.  The rank of the span of $\llist C.k.$ is at most the rank of the
span of $\llist S.k.$, the sum $h$ of the coefficients of $\llist C.k.$ is at least the
sum of the coefficients of the horizontal components of $\llist S.k.$ and the dimension of
$F$ is equal to the dimension of $X$ minus the dimension of $Z$.

As we are assuming \eqref{t_toric}$_{n-1}$, which implies \eqref{c_complexity}$_{n-1}$,
the complexity of the pair $(F,\Theta)$ is non-negative.  Thus $h<d$ so that there is an
index $i$ such that every component of $S_i$ is vertical.  In particular at least one
component $P$ of $\Delta$ is vertical.
\end{proof}

\begin{lemma}\label{l_decrease_dim} Let $(X,\Delta)$ be a divisorially log terminal 
pair, where $X$ is a $\mathbb{Q}$-factorial projective variety and $-(K_X+\Delta)$ is nef.
Let $A_0$ be an ample divisor and let $0\leq \Delta_0\leq \Delta$ be a divisor such that
$K_X+A_0+\Delta_0$ has numerical dimension $k$.  Suppose $\Delta$ has a component $P$
which is vertical for the ample model $\phi\colon\rmap X.Z_0.$ of $K_X+A_0+\Delta_0$.

Then there is an ample divisor $A_1$ and a divisor $0\leq \Delta_1 \leq \Delta$ such that
$K_X+A_1+\Delta_1$ is pseudo-effective and the numerical dimension of $K_X+A_1+\Delta_1$
is less than $k$.
\end{lemma}
\begin{proof} Set $M=-(K_X+\Delta)$.  Let $p$ be the coefficient of $P$ in $\Delta$.  Pick
$\lambda$ minimal so that
\begin{align*} 
K_X+A_1+\Delta_1 &=K_X+\lambda A_0+(1-\lambda)M+\lambda\Delta_0+(1-\lambda)(\Delta-pP)\\ 
                &=\lambda(K_X+A_0+\Delta_0)+(1-\lambda)(K_X+M+\Delta-pP)\\
                &=\lambda(K_X+A_0+\Delta_0)-(1-\lambda)pP
\end{align*} 
is pseudo-effective, where 
\[
A_1=\lambda A_0+(1-\lambda)M \qquad \text{and} \qquad \Delta_1=\lambda\Delta_0+(1-\lambda)(\Delta-pP).
\]
Note that $\lambda>0$.  In particular $A_1$ is ample and $0\leq \Delta_1\leq \Delta$.

Let $A_t=(1-t)A_0+tA_1$ and $\Delta_t=(1-t)\Delta_0+t\Delta_1$.  Let $Z_t$ be the ample
model of $K_X+A_t+\Delta_t$.  Note that $K_X+A_t+\Delta_t$ is a convex linear combination
of $K_X+A_0+\Delta_0$ and $-P$.  Therefore if $P$ is in the stable base locus of
$K_X+A_t+\Delta_t$ then $Z_t=Z_0$ and so $P$ is vertical over $Z_t$.  If $P$ is not in the
stable base locus and $t<1$ then \eqref{l_orientation} implies that $P$ is vertical over
$Z_t$.

By \cite[3.3.2]{HM10} we may find $\delta>0$ such that $Y=Z_t$ is independent of
$t\in (1-\delta,1)$ and there is a contraction morphism $f\colon\map Y.Z_1.$.

By what we just proved $P$ is vertical for $Y$.  On the other hand \eqref{l_orientation}
implies that $P$ is horizontal for $Z_1$.  Thus $f$ is not birational and so the dimension
of $Z_1$ is less than the dimension of $Y$.  In particular the numerical dimension of
$K_X+A_1+\Delta_1$ is less than $k$.
\end{proof}

\begin{lemma}\label{l_minimal} Let $(X,\Delta)$ be a divisorially log terminal pair
where $X$ is a $\mathbb{Q}$-factorial projective variety and $M=-(K_X+\Delta)$ is nef.
Let $\delta>0$ be any positive real number.  

Let $A$ be an $\mathbb{R}$-divisor and let $0\leq \Delta_0\leq \Delta$ (respectively
$\Delta_0\geq \alpha\Delta$ some $\alpha$).  Suppose that $K_X+A+\Delta_0$ is
pseudo-effective.  Let
\[
K_X+A+\Delta_0=P+N=P_{\sigma}(X,K_X+A+\Delta_0)+N_{\sigma}(X,K_X+A+\Delta_0)
\]
be Nakayama's Zariski decomposition.  We may write $N=N_0+N_1$, where every component of
$N_0$ is a component of $\Delta$ and no component of $N_1$ is a component of $\Delta$.

Then we may find $t>0$ and $\Delta_1\leq \Delta\leq (1+\delta)\Delta_1$ (respectively
$\Delta_1\geq (1-\delta)\Delta$) such that
\[
P_{\sigma}(X,K_X+A_t+\Delta_1)=tP \qquad \text{and} \qquad N_{\sigma}(X,K_X+A_t+\Delta_1)=tN_1,
\]
where
\[
A_t=tA+(1-t)M.
\]
\end{lemma}
\begin{proof} We have 
\begin{align*} 
K_X+A+\Delta_0 &\sim_{\mathbb{R}} P+N \\
K_X+M+\Delta   &\sim_{\mathbb{R}} 0.
\end{align*} 
Given $t\in (0,1]$, 
\begin{align*} 
K_X+A_t+t\Delta_0+(1-t)\Delta &=t(K_X+A+\Delta_0)+(1-t)(K_X+M+\Delta) \\
                              &\sim_{\mathbb{R}}tP+tN. 
\end{align*} 
It $t>0$ is sufficiently small then
\[
\Delta_1=t\Delta_0+(1-t)\Delta-tN_0\geq 0 \qquad \text{and} \qquad \Delta\leq (1+\delta)\Delta_1.\qedhere
\]
\end{proof}

\begin{proof}[Proof of \eqref{p_numerical}] Consider divisors of the form
$K_X+A+\Delta_0$, where $0\leq \Delta_0\leq \Delta$ and $A$ is ample.  Let $k$ be the
minimum of the numerical dimension of pseudo-effective divisors of this form.

Then $k<\infty$, since we can always pick $A$ so that $K_X+A$ is ample.  Pick
$0\leq \Delta_0\leq \Delta$ and an ample divisor $A$ such that $K_X+A+\Delta_0$ has
numerical dimension $k$.  

Suppose that $k$ is bigger than the complexity of $(X,\Delta)$.  By \eqref{l_minimal} we
may assume that no component of $\Delta$ is a component of $N$ and that the complexity of
$(X,A+\Delta_0)$ is less than $k$.  By \eqref{l_fibration} we may find a component $P$ of
$\Delta$ which is vertical for the ample model of $(X,A+\Delta_0)$.  Then by
\eqref{l_decrease_dim} we can find an ample divisor $A_1$ and a divisor
$0\leq \Delta_1 \leq \Delta$ such that $K_X+A_1+\Delta_1$ is pseudo-effective and the
numerical dimension is less than $k$, a contradiction. \end{proof}

\begin{proof}[Proof of \eqref{c_numerical}] Consider divisors of the form
$K_X+A+\Delta_0$, where $0\leq \Delta_0\leq \Delta$, $A$ is ample and no component of
$N_{\sigma}(X,K_X+A+\Delta_0)$ is a component of $\Delta_0$.  Let $k$ be the minimal
numerical dimension of pseudo-effective divisors of this form.  As
\[
c(X,\Delta)\leq \gamma(X,\Delta)<2,
\]
by \eqref{p_numerical} we have $k\leq 1$.

Pick $0\leq \Delta_0\leq \Delta$ and an ample divisor $A$ such that $K_X+A+\Delta_0$ has
numerical dimension $k$.  By \eqref{l_minimal} we may assume that that the absolute
complexity of $(X,A+\Delta_0)$ is less than $\gamma_0$.  Pick $\delta>0$ such that
$A+\delta\Delta_0$ is ample.  Replacing $A$ by $A+\delta\Delta_0$ and $\Delta_0$ by
$(1-\delta)\Delta_0$ we may assume that $(X,A+\Delta_0)$ is kawamata log terminal.

Let $\phi\colon\rmap X.Z.$ be the ample model of $K_X+A+\Delta_0$ and let
$f\colon\rmap X.Y.$ be a log terminal model of $K_X+A+\Delta_0$.  Then there is a
contraction morphism $g\colon\map Y.Z.$.  If $\Gamma=f_*(\Delta_0+A)$ then
$\gamma(Y,\Gamma) <\gamma_0$.

Suppose that $k=1$, that is, suppose $Z$ is a curve.  If $F$ is a general fibre of $g$ and
$\Theta=\Gamma|_F$ then $(F,\Theta)$ is log canonical and $K_F+\Theta$ is numerically
trivial.  The natural map which assigns to a divisor on $Y$ its restriction to $F$ has a
non-trivial kernel, since $F$ restricts to zero.  Therefore the dimension $r$ of the span
of the components of $\Theta$ is at most $\rho(Y)-1$.  \eqref{l_decrease_dim} implies that
every component of $\Gamma$ dominates $Z$.  Therefore the sum $t$ of the coefficients of
$\Theta$ is at least the sum $d$ of the coefficients of $\Delta$.  Hence
\begin{align*} 
c(F,\Theta)&\leq \dim F+r-t\\
           &\leq (\dim Y-1)+(\rho(Y)-1)-d\\
           &=\gamma(Y,\Gamma) -2\\
           &<0,
\end{align*} 
a contradiction.  Hence $k=0$ and we may apply \eqref{l_mori}. 
\end{proof}
\section{Reduction to Mori dream spaces}
\label{s_reduction}

\begin{lemma}\label{l_up}  Let $X$ be a $\mathbb{Q}$-factorial projective variety.  
Suppose that $(X,\Delta)$ is a divisorially log terminal pair and $-(K_X+\Delta)$ is nef.

Suppose that we may find a big and nef $\mathbb{Q}$-divisor divisor $A$ and a kawamata log
terminal pair $(X,\Delta_0)$ such that $K_X+\Delta_0+A \sim_{\mathbb{R}} N\geq 0$ has
numerical dimension zero.

Then we may find a divisorially log terminal modification $\pi\colon\map Y.X.$ of
$(X,\Delta)$, a big and nef $\mathbb{Q}$-divisor $B$ and a kawamata log terminal pair
$(Y,\Gamma_1)$ such that $K_Y+\Gamma_1+B \sim_{\mathbb{R}} L\geq 0$ has numerical
dimension zero, $\Gamma_1$ and $L$ have no common components and no non kawamata log
terminal centre of $(Y,\Gamma)$ is contained in the support of $L$.
\end{lemma}
\begin{proof} \eqref{l_minimal} implies that replacing $\Delta_0$ and $N$ we may assume
that $\Delta_0$ and $N$ have no common components.  Pick $t>0$ such that $\rfdown tN.=0$.
Let $\pi\colon\map Y.X.$ be a divisorially log terminal modification of $(X,\Delta+tN)$.
Then $\pi$ has finitely many exceptional divisors.  If $E$ is an exceptional divisor which
is not a log canonical place of $(X,\Delta)$ then $E$ is not a log canonical place of
$(X,\Delta+\delta N)$ for $\delta>0$ sufficiently small.  Thus replacing $t$ by $\delta>0$
sufficiently small, we may assume that $\pi$ is also a divisorially log terminal
modification of $(X,\Delta)$.

If we write
\[
K_Y+\Gamma=\pi^*(K_X+\Delta),
\]
then $(Y,\Gamma)$ is divisorially log terminal and $-(K_Y+\Gamma)$ is nef.  We may also
write
\[
K_Y+\Gamma_0=\pi^*(K_X+\Delta_0) \qquad \text{and} \qquad B=\pi^*A.
\]
If $L=\pi^*N$ then $\mult_EL>0$ for every exceptional divisor $E$ whose centre $V$ is
contained in the support of $N$.  As
\[
K_Y+\Gamma_0+B \sim_{\mathbb{R}} L,
\]
by \eqref{l_minimal} we may find $B_1$ big and nef and
$\Gamma_1\geq (1-\delta)\Gamma\geq 0$ such that
\[
K_Y+\Gamma_1+B_1 \sim_{\mathbb{R}} tL_1,
\]
where $L=L_0+L_1$ and $L_1$ has no common components with $\Gamma$.  But then no non
kawamata log terminal centre of $(Y,\Gamma)$ is contained in the support of $L$.
\end{proof}

\begin{lemma}\label{l_less} Let $X$ be a $\mathbb{Q}$-factorial projective toric variety.  
Let $B\geq 0$ be an $\mathbb{R}$-Cartier divisor whose support contains all but one
invariant divisor.  Let $\nu$ be a valuation which is not toric.  

Then we may find a divisor $0\leq B' \sim_{\mathbb{R}} B$ such that $\nu(B')>\nu(B)$
whilst $\mu(B')\leq \mu(B)$ for every toric valuation $\mu$.

If further $(X,B)$ is a log canonical pair such that every log canonical place is toric
then we may pick $B'$ such that $(X,B')$ is log canonical and the only log canonical
places of $(X,B')$ are toric valuations.  
\end{lemma}
\begin{proof} We prove the first statement.  

Let $\pi\colon\map Y.X.$ be a birational morphism of toric varieties.  As the support of
$B$ contains every invariant subset of codimension at least two, it follows that
$\pi^*B\geq 0$ is an $\mathbb{R}$-Cartier divisor whose support contains all but one
invariant divisor.  Thus we are free to replace $X$ by $Y$ and $B$ by $\pi^*B$.  In
particular we may assume that $X$ is smooth.  We are also free to replace $B$ by a
multiple.

Let $W$ be the centre of $\nu$ and let $V$ be the smallest invariant subset of $X$ which
contains $W$.  If $W=V$ then let $\pi\colon\map Y.X.$ blow up $V$.  Replacing $X$ by $Y$
and repeating this procedure finitely many times we reduce to the case $W\neq V$.
\cite[1.2]{Tevelev07} implies that we may find a birational morphism of toric varieties
$\map V'.V.$ such that the strict transform $W'$ of does not contain any invariant subsets
of $V'$.  Let $\map Y.X.$ be a birational morphism of toric varieties, which is an
isomorphism at the generic point of $V$, such that if $U$ is the strict transform of $V$
then the birational morphism $\map U.V.$ factors through $\map V'.V.$.
\cite[1.2]{Tevelev07} implies that, replacing $X$ by $Y$, we may reduce to the case when
$W$ does not contain any invariant subsets.

Note that by \eqref{l_missing} we may find a divisor $C\geq 0$ supported on the invariant
components of $B$ not containing $V$ such that $A=C|_V$ is very ample and we can lift
elements of the linear system $|A|$.

Pick a birational morphism $\pi\colon\map Y.X.$ of toric varieties such that the mobile
part of $\pi^*C$ is base point free.  Replacing $X$ by $Y$ we may assume that the mobile
part $B_0$ of $C$ is base point free.  Note that $B_0$ and $C$ are the same in a
neighbourhood of $V$.  Replacing $B$ by a multiple we may assume $B_1=B-B_0\geq 0$.  

Let $f\colon\map X.Z.$ be the contraction morphism associated to $B_0$.  Then the
restriction of $f$ to $V$ is an isomorphism.  As $W$ does not contain any invariant
subsets, it follows that $f(W)$ does not contain any invariant subsets and so $f(W)$ does
not contain the image of any invariant subvariety of $X$.

As $B_0$ is the pullback of very ample divisor from $Z$, we may pick
$0\leq B_0'\sim_{\mathbb{R}} B_0$ such that $B'_0$ contains $W$ and
$\mu(B_0')\leq \mu(B_0)$ for all toric valuations $\mu$.  If $B'=B'_0+B_1$ then
$0\leq B' \sim_{\mathbb{R}} B$, $\nu(B')>\nu(B)$, whilst $\mu(B')\leq \mu(B)$ for every
toric valuation $\mu$.

Now suppose that $(X,B)$ is log canonical.  If $B_t=tB'+(1-t)B$ then
$0\leq B_t\sim_{\mathbb{R}} B$, $\nu(B_t)>\nu(B)$ for $t>0$, $(X,B_t)$ is log canonical if
$t$ is sufficiently small and the only log canonical places are toric valuations.
\end{proof}

\begin{lemma}\label{l_dream} Assume \eqref{t_toric}$_{n-1}$.  

Let $X$ be a $\mathbb{Q}$-factorial projective variety of dimension $n$.  Suppose that
$(X,\Delta)$ is a divisorially log terminal pair and $-(K_X+\Delta)$ is nef.

If the complexity of $(X,\Delta)$ is less than one then $X$ is of Fano type.  In
particular $X$ is a toric variety.
\end{lemma}
\begin{proof} \eqref{p_numerical} implies that we may find a big and nef
$\mathbb{Q}$-divisor divisor $A$ and a kawamata log terminal pair $(X,\Delta_0)$ such that
$K_X+\Delta_0+A \sim_{\mathbb{R}} N\geq 0$ has numerical dimension zero.

By \eqref{l_up} possibly replacing $X$ by a higher model we may assume that $\Delta$ and
$N$ have no common componentd and that no non kawamata log terminal centre of $(X,\Delta)$
is contained in the support of $N$.  In particular we may pick $\epsilon>0$ such that no
non kawamata log terminal centre of $(X,\Delta+\epsilon N)$ is contained in $N$.

Let $\pi\colon\rmap X.Y.$ be a log terminal model of $(X,\Delta_0+A)$.  Set $B=\pi_*A$,
$\Gamma_0=\pi_*\Delta_0$ and $\Gamma=\pi_*\Delta$.  Then $(Y,\Gamma_0+B)$ is a kawamata
log terminal pair, $\Gamma_0+B$ is big and $K_Y+\Gamma_0+B$ is numerically trivial so that
$Y$ is of Fano type.  In particular \eqref{t_mdscase} implies that $Y$ is a toric variety.
If $\pi$ does not contract any divisors then $N=0$, $K_X+\Delta_0+A$ is numerically
trivial and so $\pi$ is an isomorphism.

Pick an exceptional divisor $E$ of $\pi$ and let $\nu$ be the corresponding valuation.  If
$\nu$ is a toric valuation of $Y$ for every exceptional divisor $E$ then $X$ is log Fano.
So we may assume that $\nu$ is not toric.  $E$ is a component of $N$ and so $\nu$ is not a
log canonical place of $(Y,\Gamma)$.  \eqref{l_less} implies that we may find
$0\leq \Gamma' \sim_{\mathbb{R}} \Gamma$ such that $\nu(\Gamma')>\nu(\Gamma)$,
$(Y,\Gamma')$ is log canonical and the only log canonical places are toric valuations.
Let $\Delta'$ be the strict transform of $\Gamma'$.  Then certainly $(X,\Delta')$ is log
canonical outside of the support of $N$, since the indeterminancy locus of $\pi$ is
contained in the support of $N$.

Let $\Delta_s=s\Delta'+(1-s)\Delta$.  As $N$ contains no non kawamata log terminal centre
of $(X,\Delta)$, it follows that $N$ contains no non kawamata log terminal centre of
$(X,\Delta_s+\epsilon N)$ for $s$ sufficiently close to zero.  In particular
$(X,\Delta_s)$ is log canonical for $s$ sufficiently close to zero.  Replacing
$(Y,\Gamma')$ by $(Y,\Gamma_s=\pi_*\Delta_s)$ for $s$ sufficiently close to zero, we may
assume that $(X,\Delta')$ is log canonical and $N$ contains no non kawamata log terminal
centre of $(X,\Delta'+\epsilon N)$.  We may write
\begin{align*} 
K_X+\Delta&=\pi^*(K_Y+\Gamma)+F\\
K_X+\Delta'&=\pi^*(K_Y+\Gamma')+F'
\end{align*} 
where $F$ and $F'$ are exceptional.  Note that the coefficients of $F'$ are linear
functions of $s$ and so we may pick $s$ sufficiently close to zero so that
$F-F'\leq \epsilon N$.

Let
\[
E_t=t(F'-F)+(1-t)N, 
\]
Decompose $E_t$ as $E^+_t-E^-_t$, where $E^{\pm}_t\geq 0$ and $E_t^+$ and $E_t^-$ have no
common components.  If we put 
\[
\Delta_t=(1-t)\Delta_0+t\Delta'+E_t^- \qquad \text{and} \qquad A_t=(1-t)A+tM
\]
then $(X,\Delta_t)$ is kawamata log terminal for $t\in [0,1)$ and we have
\begin{align*} 
K_X+\Delta_t+A_t &=(1-t)(K_X+\Delta_0+A)+t(M+K_X+\Delta')+E_t^-\\
                 &=(1-t)(K_X+\Delta_0+A)+t(\Delta'-\Delta)+E_t^-\\
                 &\sim_{\mathbb{R}} (1-t)N+t(F'-F)+E_t^-\\
                 &=E_t+E_t^-\\
                 &=E_t^+.
\end{align*} 
By assumption $\mult_EF'<\mult_EF$, so that $E_t^-\neq 0$ for $t$ sufficiently close to
$1$.  Thus $E_t^+$ has fewer components than $N$ and we are done by induction on the
number of components of $N$.  
\end{proof}

\section{Rationality via the Cox ring}
\label{s_rational}

\begin{proposition}\label{p_quadricmds} Let $X$ be a $\mathbb{Q}$-factorial projective variety.  
Suppose that the Cox ring of $X$ is a polynomial ring modulo a single relation $Q$,
\[
\Cox(X)=\frac{k[\llist x.n.]}{\langle Q\rangle},
\]
where $Q$ and $\llist x.n.$ are homogenous elements of $\Cox(X)$.  

If the rank of the quadratic part of $Q$ is at least two then there is a proper finite
morphism $\map Y.X.$ of degree at most two, which is \'etale outside a closed subset of
codimension at least two, such that $Y$ is rational.

In particular if $A_{n-1}(X)$ has no $2$-torsion then $X$ is rational.
\end{proposition}
\begin{proof} $R=\Cox(X)$ is a multigraded ring, and this grading corresponds to the
action of a diagonalisable group $H$ on $\Spec R$.  $X$ is a GIT quotient of $\Spec R$ by
$H$.  The action of $H$ extends to the polynomial ring $k[\llist x.n.]$ and the GIT
quotient of the corresponding affine space is a toric variety $Z$ which contains $X$ as a
divisor; the relation $Q$ is homogeneous for this action.

Let $T$ be the torus of $Z$.  The monomials in the coordinate ring $k[M]$ of the torus are
Laurent monomials in the variables $\llist x.n.$ of multi-degree $0$ in the grading.

Suppose that $x_ix_j\in Q$ for $i\neq j$.  Possibly permuting the coordinates we may
assume that $x_1x_2\in Q$.  Collecting together all of the terms divisible by $x_1$, we
may write
\[
Q=x_1(x_2+q_0)+q_1,
\]
where $q_1$ is a polynomial in $x_2$, $x_3$, \dots, $x_n$.  After the homogeneous change
of variable,
\[
\map x_i..\begin{cases} x_2+q_0 & \text{if $i=2$} \\
                        x_i     & \text{otherwise},
\end{cases}
\]
we may write 
\[
Q=x_1x_2-q
\]
where $q$ is a polynomial in $x_2$, $x_3$, \dots, $x_n$.  

As $Q$ is homogeneous and $q$ is not equal to zero, we may find a monomial $\nu\in q$ in
the variables $x_2$, $x_3$, \dots, $x_n$, with the same multi-degree as $x_1x_2$.  If we
set
\[
\mu=\frac{x_1x_2}{\nu}
\]
then $\mu$ is a Laurent monomial in the variables $\llist x.n.$ of multi-degree zero.

Therefore $\mu\in k[M]\subset k[\mathbb{Z}^n]$.  As the degree of $x_1$ is one in $\mu$ it
follows that $\mu=\mu_1$ corresponds to a primitive element $m_1$ of the lattice $M$ and
we may extend it to a basis $\llist m.p.$ of $M$, where the first coordinate of $m_i$ is
zero, for $i>1$.  If $\llist \mu.p.$ are the Laurent monomials corresponding to
$\llist m.p.$ of $M$, then $\mu_i$, $i>1$ are Laurent monomials in the variables $x_2$,
$x_3$, \dots, $x_n$.

Now let $G=\mathbb{G}_m$ act by $t$ on $\mu$ and trivially on all other basis elements.
Let $U$ be the open subset of the torus where $q\neq 0$.  On $X$ we have
\[
\mu=\frac{x_1x_2}{\nu}=\frac q{\nu}.
\]
As the RHS is invariant under the action of $G$ and the LHS is not, it follows that the
orbits of $G$ intersect $X\cap U$ in a unique point.  Thus $X$ is birational to $U/G$ and
so $X$ is rational.

Otherwise assume $x_ix_j\notin Q$, for all $i\neq j$.  Since the quadratic part of $Q$ has
rank at least two, $x_i^2$, $x_j^2\in Q$, for $i\neq j$.  Possibly permuting the
coordinates we may assume that $x_1^2$, $x_2^2\in Q$.  Rescaling we may assume that
\[
Q=x_1^2-x_2^2+q,
\]
where the quadratic part of $q$ is a polynomial in the variables $x_3$, $x_4$, \dots,
$x_n$.  

As $Q$ is homogeneous, $x_1^2$ and $x_2^2$ have the same multi-degree.  If $x_1$ and $x_2$
have the same multi-degree then both $x_1+x_2$ and $x_1-x_2$ are homogeneous and after the
change of coordinates
\[
\map x_i..\begin{cases} x_1+x_2 & \text{if $i=1$} \\
                        x_1-x_2 & \text{if $i=2$} \\
                        x_i     & \text{otherwise},
\end{cases}
\]
$x_1x_2\in Q$ so that $X$ is rational by what we already proved.

Otherwise $x_1/x_2$ is torsion of degree two.  By definition of the Cox ring, there is a
Weil divisor $D$ on $X$ such that $2D\sim 0$.  $D$ defines a proper finite morphism
$\map Y.X.$ of degree two, which is \'etale outside a closed subset of codimension at
least two.  

On the other hand, $E$ lifts to a Weil divisor on $Z$ such that $2E\sim 0$.  $E$ also
defines a proper finite morphism $\map W.Z.$ of degree two, which is also \'etale outside
a closed subset of codimension at least two, and induces the original cover $\map Y.X.$.

$W$ is a toric variety.  Moreover $W$ has the same Cox ring as $Z$ but with a grading
given by setting the class of $D$ equal to zero.  This grading corresponds to the action
of a diagonalisable group $G$ on $\Spec R$ and $W$ is a GIT quotient of $\Spec R$ by $G$.
$Y$ is a divisor in $W$, defined by the same equation as $X$ inside $Z$.  On $W$, however,
$x_1$ and $x_2$ have the same multi-degree and so $Y$ is rational by what we have already
proved.  \end{proof}

Note that $Y$ does not necessarily have the same Cox ring as $X$, since $Y$ might have
more divisors than $X$, as happens in the example in \S \ref{s_example}.
\section{An irrational example}
\label{s_example}

In this section we give an example of an irrational projective threefold $X$, together
with a log canonical pair $(X,\Delta)$ of absolute complexity $1$ such that $K_X+\Delta$
numerically trivial.  In particular the condition on torsion in $A_{n-1}(X)$ in
\eqref{p_quadricmds} is necessary.  We will construct $X$ as a $\mathbb{Z}_2$-quotient of
a conic bundle $Y$ over $T=\pr 1.\times \pr 1.$.

Pick bihomogeneous coordinates $([y_0:y_1],[z_0:z_1])$ on $T$.  If $q$ is a general
polynomial of degree $(2d,2d)$ in the monomials $y_i^2$, $z_i^2$ and $y_iz_i$,
$i\in \{\,0,1\,\}$, where $d>3$, then the zero locus of $q$ is a smooth curve $D$ which
contains none of the invariant points.  The equation
\[
x_0^2-x_1^2=q(y_0,y_1,z_0,z_1)x_2^2,
\] 
defines a divisor $Y$ inside the projectivisation $W$ of
\[
\ring {\pr 1.\times \pr 1.}.\oplus \ring {\pr 1.\times \pr 1.}.  \oplus \ring {\pr 1.\times \pr 1.}.(-d,-d).
\] 
$Y$ is a conic bundle $f\colon\map Y.T.$ over $T$.  Let $\Theta$ be the divisor on $W$ given by
the sum of the vanishing of $x_2$, together with the pullbacks of the torus invariant
divisors from $\pr 1. \times \pr 1.$, and set $\Gamma=\Theta|_Y$.

\begin{lemma}\label{l_y}\
\begin{enumerate} 
\item $(Y,\Gamma)$ is log smooth,
\item $K_Y+\Gamma=0$, and
\item $f\colon\map Y.T.$ has relative Picard number two.
\end{enumerate} 
\end{lemma}
\begin{proof} To prove (1) and (2), by adjunction it suffices to check that $(W,Y+\Theta)$
is log smooth and $K_W+Y+\Theta=0$.

Consider the linear series on $W$ spanned by $x_0^2$, $x_1^2$, and $x_2^2m$, where $m$
ranges over all monomials in $y_i^2$, $z_i^2$ and $y_iz_i$, $i\in \{\,0,1\,\}$, of degree
$(2d,2d)$.  Rescaling $x_0$ and $x_1$ we see that $Y$ is a general member of this linear
series.  On the other hand this linear series is base point free and so $Y$ is smooth and
intersects all torus invariant strata of $W$ transversely.  This gives (1) and (2).

For (3) note that the fibres of $f$ are irreducible except over the curve $D$, where the
fibres have two components.  It follows that the relative Picard number is at most two.
On the other hand, the inverse image of $D$ consists of two prime divisors
$D_1=V(x_0+x_1,q)$ and $D_2=V(x_0-x_1,q)$.  This is (3).
\end{proof}

Consider the $\mathbb{Z}_2$ action on $W$ sending
\[
\map (y_0,y_1,z_0,z_1,x_0,x_1,x_2).(y_0,-y_1,z_0,-z_1,x_0,-x_1,x_2)..
\]  
This action also defines an action on $T$ and under this action both $Y$ and $D$ are
invariant.  If $X$ is the quotient of $Y$ and $S$ is the quotient of $T$ then there is a
commutative diagram
\[
\begin{diagram}
Y   &   \rTo    &  X  \\
\dTo^f &  & \dTo^g \\
T   &   \rTo    &  S.
\end{diagram}
\]
Note that $g\colon\map X.S.$ is a conic bundle.  Note also that the action on $T$ is
toric, fixing only the four torus invariant points and so $S$ is also a toric surface with
four $A_1$ singularities.

Let $\Delta$ be the image of $\Gamma$.
\begin{proposition} \
\begin{enumerate} 
\item $(X,\Delta)$ is log canonical
\item $K_X+\Delta=0$, 
\item the absolute complexity $\gamma$ is one, and
\item $X$ is irrational.
\end{enumerate} 
\end{proposition}
\begin{proof} Note that as $\map Y.X.$ is \'etale in codimension one, (1) and (2) follow 
easily.  

The action of $\mathbb{Z}_2$ switches the divisors $D_1$ and $D_2$.  Thus $X$ has relative
Picard number $1$ over $S$ and so the absolute complexity of $(X,\Delta)$ is
\[
\gamma=\dim X+\rho(X)-d=3+3-5=1.
\]  
This is (3).

If $X$ is rational then the Griffiths component of the intermediate Jacobian must be
trivial.  If $C$ is the discriminant curve of the conic bundle $g\colon\map X.S.$ and the
Griffiths component is trivial then $C$ is hyperelliptic, trigonal, or isomorphic to a
plane quintic, cf. \cite{Shokurov83}.

$C$ is certainly not a plane quintic.  On the other hand $C$ is a smooth quotient of $D$
by the fixed point-free action of $\mathbb{Z}_2$.  It suffices to check that $D$ has no
$g_6^1$ (since if $D$ has a $g^1_3$ it has a $g^1_6$).  But it follows from a theorem of
Martens \cite{Martens96} that if $D$ has a $g^1_k$ then $k\geq 2d$.
\end{proof}

\section{Proofs}
\label{s_proofs}

\begin{proof}[Proof of \eqref{t_toric}] We proceed by induction on the dimension $n$ of
$X$.

\eqref{p_dlt} implies that we may find a divisorially log terminal model
$\pi\colon\rmap Y.X.$, such that if we write
\[
K_Y+\Gamma=\pi^*(K_X+\Delta),
\]
then $-(K_Y+\Gamma)$ is nef.  \eqref{l_dlt} implies that the complexity of $(Y,\Gamma)$ is
at most the complexity of $(X,\Delta)$.  \eqref{l_persist} implies that if $(Y,G)$ is a
toric pair then $(X,D)$ is a toric pair, where $D=\pi_*G$.  If $G\geq \sship \Gamma.$ then
$D\geq \sship \Delta.$ and if all but $\rfdown 2c.$ invariant divisors are components of
$\Gamma$ then all but $\rfdown 2c.$ invariant divisors are components of $\Delta$.
Replacing $(X,\Delta)$ by $(Y,\Gamma)$ we may assume that $X$ is a $\mathbb{Q}$-factorial
projective variety and $(X,\Delta)$ is divisorially log terminal.

\eqref{l_dream} implies that $X$ is of Fano type.  Thus $X$ is a Mori dream space and we
are done by \eqref{t_mdscase}.
\end{proof}

\begin{proof}[Proof of \eqref{c_complexity}] It it immediate from \eqref{t_toric} that 
$\rdown 2c.\geq 0$.  
\end{proof}

\begin{proof}[Proof of \eqref{c_span}] If $c<1$ then \eqref{t_toric} implies that $X$ is
toric and all but one invariant divisor is a component of $\Delta$.  On the other hand the
invariant divisors span the N\'eron-Severi group and any invariant divisor is in the span
of the other invariant divisors.
\end{proof}

\begin{proof}[Proof of \eqref{c_non}] Let $\bar k$ be the algebraic closure of $k$ and let
bars denote extension to the algebraic closure.

Then $(\bar X,\bar \Delta)$ is log canonical and $-(K_{\bar X}+\bar \Delta)$ is nef.  If
$\sum a_iS_i$ is a decomposition of $\Delta$ of complexity less than one then
$\sum a_i\bar S_i$ is a decomposition of $\bar \Delta$ of complexity less than one.
\eqref{t_toric} implies that there is a divisor $\bar D$ such that $(\bar X,\bar D)$ is
toric.  

Let $m$ be the number of invariant divisors.  Possibly reordering, we may assume that
$\llist \bar S.m-1.$ are invariant divisors.  In particular $\llist S.m-1.$ are prime
divisors.  Consider the linear system 
\[
|-(K_X+\sum_{i\leq m-1} S_i)|.
\]  
If $\bar D_m$ is the last invariant divisor on $\bar X$ then
\[
\bar D_m\in |-(K_{\bar X}+\sum_{i\leq m-1} \bar S_i)|.
\]  
Thus the linear system 
\[
|-(K_X+\sum _{i\leq m-1} S_i)|
\] 
is non-empty and we may find $S_m$ such that $K_X+D$ is linearly equivalent to zero and
$(X,D=\sum _{i\leq m} S_i)$ is log canonical.

In this case $(\bar X,\bar D)$ is toric, again by \eqref{t_toric}.  Replacing $(X,\Delta)$
by $(X,D)$ we may assume that $\Delta=D$.  In this case every component of $\bar D$ is an
invariant divisor and $\llist S.m.$ are all prime divisors.  By induction
$(S_i,(D-S)|_{S_i})$ is a toric pair.  In particular the strata of $D$ are geometrically
irreducible.

We may find $\bar\pi\colon\map \bar Y.\bar X.$ a birational morphism of toric varieties
such that $\bar Y$ is projective and there is a birational morphism to projective space,
$\bar g\colon\map \bar Y.{\gpr n.\bar k.}.$.  As the strata of $D$ are geometrically
irreducible, there is a birational morphism $\pi\colon\map Y.X.$ which extracts only
divisors of log discrepancy zero.  If we write
\[
K_Y+G=\pi^*(K_X+D),
\]
then $G$ is the sum of the strict transform of $D$ and the exceptional divisors.  It is
enough to prove that $(Y,G)$ is toric by \eqref{l_morphism}.  

Replacing $(X,D)$ by $(Y,G)$ we may assume that there is a birational morphism
$\bar f\colon\map \bar X.{\gpr n.\bar k.}.$.  Pulling back an invariant hyperplane, this
linear system is given by a sum of invariant divisors $\sum b_j\bar D_j$.  Consider the
linear system $|\sum b_jD_j|$.  This is base point free, has dimension $n$ and separates
points.  Thus we get a birational map to projective space $f\colon\map X.{\gpr n.k.}.$
such that $\bar f$ is toric.  

In particular $f$ only extracts divisors of log discrepancy zero.  \eqref{l_persist}
implies that $(X,D)$ is a toric pair.
\end{proof}

\begin{proof}[Proof of \eqref{t_form}] As \eqref{t_toric} holds in all dimensions 
\eqref{p_numerical} implies the first statement.  

Let $c$ be the complexity of $(X,\Delta)$.  Pick $\delta>0$ such that
$A+(1-\delta)\Delta_0$ is ample.  Replacing $\Delta_0$ by $(1-\delta)\Delta_0$ and $A$ by
$A+(1-\delta)\Delta_0$ we may assume that $(X,\Delta_0)$ is kawamata log terminal.  Let
$f\colon\rmap X.Y.$ be a log terminal model of $(X,A+\Delta_0)$.  Replacing $X$ by $Y$ we
may assume that $K_X+A+\Delta_0$ is nef.  In this case $K_X+A+\Delta_0$ is semiample.  Let
$f\colon\map X.W.$ be the induced model.  Then $-(K_X+\Delta_0)$ is ample over $W$.
\cite{HM05a} implies that the fibres of $f$ are rationally connected.  Thus $f$ factors
through the maximal rationally connected fibration $\rmap X.Z.$.  It follows that
\[
\dim Z\leq \dim Y\leq \nu(X,A+\Delta_0) \leq c. \qedhere
\]
\end{proof}

\begin{proof}[Proof of \eqref{t_rational}] \eqref{p_dlt} implies that we may find a
divisorially log terminal model $\pi\colon\rmap Y.X.$, such that if we write
\[
K_Y+\Gamma=\pi^*(K_X+\Delta),
\]
then $-(K_Y+\Gamma)$ is nef.  \eqref{l_dlt} implies that the absolute complexity of
$(Y,\Gamma)$ is at most the absolute complexity of $(X,\Delta)$.  Replacing $(X,\Delta)$
by $(Y,\Gamma)$ we may assume that $X$ is a $\mathbb{Q}$-factorial projective variety and
$(X,\Delta)$ is divisorially log terminal.

By \eqref{c_numerical} we may find a divisorially log terminal pair $(Y,\Gamma)$ such that
$-(K_Y+\Gamma)$ is ample, the absolute complexity is less than two and $Y$ is birational
to $X$.  Replacing $(X,\Delta)$ by $(Y,\Gamma)$ we may assume that $-(K_X+\Delta)$ is
ample.  \eqref{l_mori} implies that $X$ is a Mori dream space.  Pick
$B \sim_{\mathbb{R}} -(K_X+\Delta)$ such that $(X,B+\Delta)$ is divisorially log terminal.
Replacing $(X,\Delta)$ by $(X,B+\Delta)$ we may assume that $K_X+\Delta$ is numerically
trivial.

Let $R=\Cox(X)$ be the Cox ring of $X$, $Y=\Spec R$, and $\Gamma$ the divisor on $Y$
corresponding to $\Delta$.  Then every component of $\Gamma$ is Cartier and $K_Y$ is
Cartier.  \cite{KO12} implies that $(Y,\Gamma)$ is log canonical as $(X,\Delta)$ is log
canonical.

By \eqref{l_flipsabundance} $Y$ has a $cA_l$ singularity at the point $p$.  If $Y$ is
smooth then $X$ is a toric variety and there is nothing to prove.  Otherwise $\Cox(X)$ is
a polynomial ring modulo a single relation $Q$, where the rank of the quadratic part of
$Q$ is at least two.  Thus we may apply \eqref{p_quadricmds}.  \end{proof}

\bibliographystyle{/home/mckernan/Jewel/Tex/hamsplain}
\bibliography{/home/mckernan/Jewel/Tex/math}


\end{document}